\documentclass[final,3p,times]{elsarticle}
\usepackage{amssymb}
\usepackage{amsmath}
\usepackage{amsthm}
\usepackage{float}
\newtheorem{theorem}{Theorem}
\newtheorem{lemma}{Lemma}
\newtheorem{proposition}{Proposition}

\newtheorem{assumption}{Assumption}

\newtheorem{remark}{Remark}
\usepackage{booktabs}
\usepackage{subcaption}
\usepackage{multirow}
\usepackage{hyperref}
\usepackage{array}

\journal{Journal of Computational and Applied Mathematics}

\begin{document}

\begin{frontmatter}

\title{Asynchronous Stochastic Block Projection Algorithm for Solving Linear Systems under Predefined Communication Patterns\tnoteref{fund}}

\author[sdust]{Yanchen~Yin}
\author[sdust]{Yongli~Wang\corref{cor1}}

\cortext[cor1]{Corresponding author. Email: \texttt{wangyongli@sdkd.net.cn}}

\tnotetext[fund]{This work was supported by the National Natural Science Foundation of China Joint Fund Key Project under Grant U22B2049.}

\address[sdust]{College of Mathematics and Systems Science,\\
                Shandong University of Science and Technology,\\
                Qingdao 266590, China}

\begin{abstract}
This paper proposes an event-triggered asynchronous distributed randomized block Kaczmarz projection (ER-AD-RBKP) algorithm for efficiently solving large-scale linear systems in resource-constrained and communication-unstable environments. The algorithm enables each agent to update its local state estimate independently and engage in communication only when specific triggering conditions are satisfied, thereby significantly reducing communication overhead. At each iteration, agents perform projections using randomly selected partial local data blocks to lower per-iteration computational costs and enhance scalability. By defining events that ensure strong connectivity in the communication graph, we derive the sufficient conditions for global convergence under a probabilistic framework, proving that the algorithm converges exponentially in expectation as long as no extreme events (e.g., permanent agent disconnection) occur. Besides, for inconsistent systems, auxiliary variables are incorporated to transform the problem into an equivalent consistent formulation, and theoretical error bounds are derived. Moreover, we implement the ER-AD-RBKP algorithm in an asynchronous communication environment built on ROS2, a distributed middleware framework for real-time robotic systems. We evaluate the algorithm under various settings, including varying numbers of agents, neighborhood sizes, communication intervals, and failure scenarios such as communication disruptions and processing faults. Experimental results demonstrate the robust performance of the proposed algorithm in terms of computational efficiency, communication cost, and system resilience, highlighting its strong potential for practical applicability in real-world distributed systems. 
\end{abstract}


\end{frontmatter}


\section{introduction}
Distributed computing technologies have attracted growing interest, with researchers developing distributed algorithms to solve large-scale computational problems. These algorithms have been applied in various fields, including optimization~\cite{Yang2019A}, deep learning~\cite{10.1145/3320060, 10.1145/3363554}, robotic systems~\cite{halsted2021surveydistributedoptimizationmethods}, and power systems~\cite{7990560}. Among these applications, solving linear systems of the form $Ax = b$ remains one of the most fundamental and widely studied problems. Traditional methods such as the Jacobi method, Gauss-Seidel method, Conjugate Gradient (CG), and Multigrid method often perform poorly on large-scale systems. This has prompted researchers to explore distributed approaches for tackling such problems.

A commonly adopted strategy in decentralized distributed settings is row-wise partitioning. Given a matrix $A \in \mathbb{R}^{m \times n}$ and a vector $b \in \mathbb{R}^m$, the global system $Ax = b$ is decomposed into $N$ local subproblems $A_i x = b_i$, each handled by an individual agent. Each agent $i$ holds a local dataset $\begin{pmatrix}A_i & b_i\end{pmatrix}$, with the aggregated data expressed as
\[
\begin{pmatrix}A & b\end{pmatrix} = \begin{pmatrix}A_1 & b_1 \\ \cdots & \cdots \\ A_N & b_N\end{pmatrix}.
\]
Each agent maintains an $n$-dimensional local state vector $x_i$ as its estimate of the global solution. Through information exchange with neighboring agents—according to a predefined communication topology—each agent integrates received states with its own and updates $x_i$ via a customized projection-based iterative scheme based on local data. The goal is to ensure that all agents’ local state vectors converge to a solution of the global system $Ax = b$.

Numerous studies have adopted the row-partitioned strategy to address the distributed solution of large-scale linear systems. A series of works by Liu and Mou systematically established the theoretical foundation of projection-based distributed algorithms for linear systems, covering convergence under fixed communication graphs~\cite{Mou2013A,Mou2013B,Liu2013An}, necessary and sufficient conditions for exponential convergence and time-varying graph analysis~\cite{Liu2014Stability,Mou2015A,Mou2015B}, and unified continuous-discrete modeling~\cite{Liu2016A}. Subsequent research extended this framework in multiple directions.For example, Wang et al. improved the framework in terms of initialization and parameter design, discussing convergence to the minimum-norm solution~\cite{Wang2016A,Wang2017Improvement,Wang2017A}; Hu et al. utilized historical iteration information to achieve acceleration~\cite{Hu2020Accelerated}; Gade et al. proposed a finite-time algorithm from a privacy perspective~\cite{Gade2020A}; Yi et al. extended the approach to stochastic communication graphs~\cite{Yi2020Distributed}; Zhu et al. investigated Byzantine fault tolerance\cite{Zhu2023A}; while Alaviani et al. introduced randomized Krasnoselskii–Mann iteration for adaptation to stochastic networks~\cite{Alaviani2018A,9144422}. Different from projection-based distributed algorithms, non-projection-based continuous or discrete gradient-type methods derive convergence and rate results through Lyapunov analysis~\cite{Yang2015A,8028633,Huang2021Distributed}, with extensions to bilevel networks, sparse information scenarios~\cite{Huang2024Distributed,Wang2020Scalable}, and others.

The vast majority of the aforementioned works are based on synchronous distributed algorithms. However, in the distributed solving of large-scale linear systems, communication processes involve transmitting substantial data among agents. To maintain synchronization, agents often must wait for communication to complete, resulting in significant idle time due to communication latency. To reduce communication overhead,~\cite{7943430,7587861,Yin2020Securely,Liu2020Communication-Efficient} employed partial data transmission to minimize waiting time, while~\cite{Wang2018Communication-efficient} proposed a low-communication-cost algorithm based on gradient descent and approximate Newton methods for linear systems with Laplacian sparse structures, achieving faster convergence. Nevertheless, these works still rely on synchronous distributed algorithms.

To address the communication latency in synchronous algorithms, researchers have begun investigating asynchronous distributed algorithms for solving large-scale linear systems and achieved notable progress. Asynchronous methods allow processors to proceed without waiting for updates from others, reducing idle time and enhancing scalability~\cite{10.1145/2814566, Liu2014AnAP}. Some algorithms have been developed for general symmetric positive-definite matrices~\cite{10.1145/2814566} and consistent linear systems~\cite{Liu2018Asynchronous, SAHU20231}. Randomized approaches, such as asynchronous block Kaczmarz and coordinate descent methods, have been proven to achieve linear convergence rates with near-linear speedup as the number of processors increases~\cite{Liu2014AnAP, doi:10.1137/19M1251643}. Compared to their synchronous counterparts, accelerated versions of these asynchronous algorithms attain optimal complexity and faster convergence ~\cite{hannah2018a2bcdasynchronousacceleratedblock, Lee2013EfficientAC}. These advancements have improved solvers for various linear algebra problems, including symmetric diagonally dominant systems and overdetermined equations~\cite{Lee2013EfficientAC}. Overall, asynchronous algorithms demonstrate significant potential for efficient distributed solutions to large-scale linear systems.

Despite this, existing asynchronous distributed algorithms still face two key challenges in practical applications. First, when the problem dimension n is large, the full-data projection operations required in each iteration lead to a computational complexity of $O(n^2)$, making it difficult to meet high-frequency update demands even with preprocessing.
Second, many methods adopt an “iterate-then-broadcast” mechanism, generating redundant information exchange that increases network load in large-scale systems, potentially undermining the parallel advantages of asynchrony.

To address these issues, this paper proposes an event-triggered asynchronous distributed randomized block Kaczmarz projection algorithm. The key contributions are as follows:

 Firstly, a Randomized block projection strategy is introduced, where each agent computes its updates using only partial local data per iteration, significantly improving computational efficiency and making the algorithm more suitable for industrial devices with limited computing power. Through rigorous theoretical analysis, the exponential convergence of the algorithm is proved under appropriate conditions for consistent systems.
  
 Secondly, an event-triggered communication mechanism is proposed and agents exchange information only when specific conditions are met, alleviating the system burden caused by high-frequency communication and enhancing overall communication efficiency. 
  
 Thirdly, for inconsistent systems, the problem is reformulated as a consistent augmented system by introducing auxiliary variables. This transformation introduces no additional communication or computational overhead and preserves the convergence properties of the original asynchronous framework, ensuring robustness in broader problem settings. 
  
Furthermore,  a theoretical analysis is conducted on the approximation error between the solution of the augmented system and the true least-squares solution. The quantitative upper bound is derived as a function of the system condition number and the design parameter. With appropriate parameter tuning, the error can be tightly controlled, and experimental results validate the theoretical estimates.
  
 Finally, the algorithm is deployed asynchronously on the ROS2 platform and extensively tested under various configurations, including varying agent numbers, neighbor sizes, communication intervals, and node failure scenarios. The numerical results confirm its practicality for real-world deployment in distributed environments, demonstrating robust and scalable performance across key metrics. 
  

In this paper, all vectors are assumed to be column vectors unless otherwise specified. Let $n$ and $m$ be arbitrary integers. We denote $\mathbb{R}^n$ as the $n$-dimensional Euclidean space, and $\mathbb{R}^{m \times n}$ as the space of real $m \times n$ matrices. Unless stated otherwise, $\|\cdot\|_2$ and $\|\cdot\|_\infty$ denote the standard vector 2-norm and infinity norm, respectively, or their corresponding matrix operator norms. The symbol $\otimes$ denotes the Kronecker product.

For a matrix $A \in \mathbb{R}^{m \times n}$,  $A^T$ denotes its transpose, and $A^\dagger$ the Moore–Penrose pseudoinverse of $A$. Let $\{\alpha_1^T, \cdots, \alpha_m^T\}$ be the set of row vectors of $A$, and define the row index set as $\mathbf{m} = \{1, 2, \cdots, m\}$. The row space of $A$ is defined as $\text{Row}(A) = \text{span}\{\alpha_1, \cdots, \alpha_m\}$, and its orthogonal complement is denoted by $\text{Row}_\perp(A)$, which corresponds to the null space of $A$ under the standard inner product.

For distributed processing, we partition $A$ into $N$ row blocks with $A_i = \begin{bmatrix} \alpha_{m_{i-1}+1}, \ \cdots, \ \alpha_{m_i} \end{bmatrix}^T$, $i=1,2,\dots,N$. The corresponding row index subset is defined as $\mathbf{m}_i = \{m_{i-1}+1, \cdots, m_i\}$, such that $\overset{N}{\bigcup\limits_{i=1}} \mathbf{m}_i = \mathbf{m}$.

The remainder of the paper is organized as follows. Section I introduces the problem. Section II describes the proposed algorithm. Section III presents the main theoretical results. Section IV reports numerical evaluations. All proofs are provided in the Appendix.

\section{algorithm}

This section first presents the asynchronous distributed randomized block projection algorithms for both consistent and inconsistent linear systems, followed by a discussion on communication cost and communication strategies in asynchronous settings.

We begin by introducing the basic notation. Let $\{x_i(t_{i,k})\}_{k=0}^{+\infty}$ denote the solution estimate sequence of agent $i$, where $t_{i,k}$ is the time at which agent $i$ completes its $k$-th local iteration. Here, $x_i(t_{i,k})$ is the corresponding solution estimate, and $x_i(t_{i,0})$ is the initial estimate. To capture asynchronous inter-agent information exchange, the solution estimate of agent $i$ is modeled as a piecewise constant function of time:
\[
x_i(t) = x_i(t_{i,k}), \quad \text{for } t \in \left[t_{i,k}, t_{i,k+1}\right).
\]

\subsection{The Consistent Case}

According to~\cite{Liu2018Asynchronous}, when the linear system $Ax = b$ is consistent, the asynchronous distributed projection algorithm can be represented as a discrete-time system with bounded random delays:
\begin{equation}
\begin{aligned}
& A_i x_i(t_{i,0}) = b_i,\\
& w_i(t_{i,k}) = \sum_{j \in \mathcal{N}_i(t_{i,k})} x_j\left(t_{i,k} - \delta_{j,i}(t_{i,k})\right),\\
& x_i(t_{i,k+1}) = x_i(t_{i,k}) - P_i\left(x_i(t_{i,k}) - \frac{1}{d_i(t_{i,k})} w_i(t_{i,k})\right),
\end{aligned}
\label{eq:initial-iteration}
\end{equation}
where $P_i$ denotes the projection onto the null space of $A_i$, i.e., $P_i = I_n - A_i^\dagger A_i$, as defined in~\cite{Liu2018Asynchronous} and commonly used in projection-based iterative methods. The set $\mathcal{N}_i(t_{i,k})$ represents the neighborhood of agent $i$ at time $t_{i,k}$; agent $j$ belongs to this set if agent $i$ uses its information at iteration $k+1$.

The delay $\delta_{j,i}(t_{i,k})$ accounts for the time lag in receiving agent $j$’s information. If agent $i$ uses $x_j(t_{j,h_{j,i}^k})$ from agent $j$’s $h_{j,i}^k$-th iteration, then $\delta_{j,i}(t_{i,k}) = t_{i,k} - t_{j,h_{j,i}^k}$. The term $d_i(t_{i,k}) = |\mathcal{N}_i(t_{i,k})|$ denotes the number of neighbors of agent $i$ at time $t_{i,k}$.

Note that each agent includes its own estimate in the neighborhood set, i.e., $i \in \mathcal{N}_i(t_{i,k})$ and $t_{i,h_{i,i}^k} = t_{i,k}$, which implies $\delta_{i,i}(t_{i,k}) = 0$. If agent $j$ has transmitted multiple updates to agent $i$ before iteration $k+1$, agent $i$ selects only the most recent estimate received from agent $j$.

When the problem dimension $n$ is large, the per-iteration computational complexity remains $O(n^2)$ even if the projection operator $P_i$ is precomputed in advance. To remove the initialization constraint on $x_i(t_{i,0})$, we reformulate the update into a block Kaczmarz-style iteration based on $A_i$:
\begin{equation}
x_i(t_{i,k+1}) = w_i(t_{i,k}) + A_i^\dagger \left( b_i - A_i w_i(t_{i,k}) \right).
\label{eq:modified-iteration}
\end{equation}
With $A_i^\dagger$ precomputed, the per-iteration computational cost for agent $i$ is reduced to $O(m_i \, n)$. To further improve efficiency, a smaller $m_i$ is desired. However, a large total row dimension $m$ then necessitates a larger number of agents $N$, which may introduce heavier communication loads during distributed deployment.

To address this issue at the algorithmic level, we adopt a randomized block Kaczmarz approach. Specifically, at each iteration, only a randomly selected subset of rows from $\big(A_i, \ b_i\big)$ is used for computation. Let $\left( A_i(t_{i,k}),\ b_i(t_{i,k}) \right)$ denote the selected sub-block at iteration $k$. The update rule in~\eqref{eq:modified-iteration} is thus modified as:
\begin{equation}
x_i(t_{i,k+1}) = w_i(t_{i,k}) + A_i^\dagger(t_{i,k}) \left( b_i(t_{i,k}) - A_i(t_{i,k}) w_i(t_{i,k}) \right).
\label{eq:consistent-update}
\end{equation}
Let $m_i(t_{i,k})$ denote the row size of the random sub-block. This update reduces the computational complexity to $O\left(m_i(t_{i,k}) \, n\right)$, significantly lowering both per-iteration computation and the scale of communication. The method is especially suitable for scenarios where $A_i$ is generated dynamically or accessed incrementally. The convergence of~\eqref{eq:consistent-update} will be formally established in \textbf{Theorem 2}.

To facilitate theoretical analysis, the iteration in~\eqref{eq:consistent-update} can be reformulated as:
\begin{equation}
x_i(t_{i,k+1}) = P_i(t_{i,k}) \left( \frac{1}{d_i(t_{i,k})} w_i(t_{i,k}) \right) + A_i^\dagger(t_{i,k}) b_i(t_{i,k}),
\label{eq:consistent-theory-form}
\end{equation}
where $P_i(t_{i,k})=I_n - A_i^\dagger(t_{i,k}) A_i(t_{i,k})$.

\subsection{The Inconsistent Case}

Let $x^* = A^\dagger b$ denote the exact least-squares solution to the inconsistent system $Ax=b$. Starting from the update rule in~\eqref{eq:consistent-update}, we derive:
\begin{equation}
\begin{aligned}
& x_i(t_{i,k+1}) - x^* =  \frac{1}{d_i(t_{i,k})}w_i(t_{i,k}) - x^*  - A_i^\dagger(t_{i,k}) A_i(t_{i,k})\left( \frac{1}{d_i(t_{i,k})}w_i(t_{i,k}) - x^* \right) \\
& + A_i^\dagger(t_{i,k}) b_i(t_{i,k}) - A_i^\dagger(t_{i,k}) A_i(t_{i,k}) x^*  = P_i(t_{i,k}) \left( \frac{1}{d_i(t_{i,k})}w_i(t_{i,k}) - x^* \right) + \epsilon_i(t_{i,k}),
\end{aligned}
\label{eq:system-dynamics}
\end{equation}
where  $\epsilon_i(t_{i,k}) = A_i^\dagger(t_{i,k})\left( b_i(t_{i,k}) - A_i(t_{i,k}) x^* \right)$ is defined as the disturbance term.

If the system $Ax = b$ is consistent, then $\epsilon_i(t_{i,k}) = 0$, and \eqref{eq:system-dynamics} reduces to a time-varying discrete-time linear system with delay. In the Appendix, we show that this system can be equivalently reformulated as a delay-free linear system, enabling convergence analysis via the spectral radius of the corresponding state transition matrix.

However, if $Ax = b$ is inconsistent, then there exists at least one $\epsilon_i(t_{i,k}) \ne 0$, which introduces a persistent disturbance into~\eqref{eq:system-dynamics}. Consequently, the delay-free reformulation derived from~\eqref{eq:consistent-update} becomes a disturbed linear system, leading to oscillatory behavior and degraded convergence performance.

To eliminate the disturbance term $\epsilon_i(t_{i,k})$, we can reformulate the original system $Ax = b$, by introducing an auxiliary vector $y \in \mathbb{R}^m$, as:  
\begin{equation}
\begin{pmatrix} A & \Lambda \end{pmatrix}
\begin{pmatrix} x \\  \\ y \end{pmatrix} = b,
\label{eq:augmented-system}
\end{equation}
where $\Lambda = \lambda I_m$ is a tunable scaling matrix with design parameter $\lambda > 0$.

For any index subset $\mathcal{J} \subseteq \mathbf{m}$, we define $E_{\mathcal{J}} \in \mathbb{R}^{m_i(t_{i,k}) \times m}$ as a row selector matrix that has the identity submatrix $I_{\mathcal{J}}$ in the rows indexed by $\mathcal{J}$ and zeros elsewhere. Let $A’ = \begin{pmatrix} A & \lambda I_m \end{pmatrix}$, the distributed algorithm is then tasked with solving the consistent augmented system in~\eqref{eq:augmented-system}.

In this case, each agent $i$ is assigned an augmented local matrix $A_i’ = \begin{pmatrix} A_i & \lambda E_{\mathbf{m}_i} \end{pmatrix}$ with $b_i \in \mathbb{R}^{m_i}$ and the iteration framework follows the same structure as in~\eqref{eq:consistent-update}. At the $k$-th iteration, agent $i$ randomly selects a sub-block $\left( A’_i(t_{i,k}),\ b_i(t_{i,k}) \right)$, where $A’_i(t_{i,k}) \in \mathbb{R}^{m_i(t_{i,k}) \times (n+m)}$.

Let $w_i^z(t_{i,k}) = \frac{1}{d_i(t_{i,k})} \sum_{j \in \mathcal{N}_i(t_{i,k})} z_j\left( t_{i,k} - \delta_{j,i}(t_{i,k}) \right)$ denote the average of the neighbor state vectors, where $z_i(t) \in \mathbb{R}^{n+m}$ represents the augmented state. The asynchronous randomized block update for the augmented system is then given by:
\begin{equation}
z_i(t_{i,k+1}) = w_i^z(t_{i,k}) + {A_i’}^\dagger(t_{i,k}) \left( b_i(t_{i,k}) - A_i’(t_{i,k}) w_i^z(t_{i,k}) \right).
\label{eq:augmented-kaczmarz}
\end{equation}

Equation~\eqref{eq:augmented-kaczmarz} can be further simplified. Let $A_i’(t_{i,k}) = \left( A_i(t_{i,k}),\ \lambda E_{\mathcal{J}(t_{i,k})} \right)$, where $\mathcal{J}(t_{i,k}) \subseteq \mathbf{m}_i$ is the row index subset selected by agent $i$ at iteration $k+1$. The definition of $A_i(t_{i,k})$ remains consistent with that in the consistent case.

Let $z_i(t_{i,k}) = \begin{pmatrix} x_i(t_{i,k})^\top,\ y_i(t_{i,k})^\top \end{pmatrix}^\top$ denote the augmented state, and $y_{i,\mathcal{J}(t_{i,k})}(t_{i,k})$ represent the components of $y_i(t_{i,k})$ corresponding to indices in $\mathcal{J}(t_{i,k})$. Define the following quantities:
\[
\begin{aligned}
r_i(t_{i,k}) &= b_i(t_{i,k}) - A_i(t_{i,k}) w_i(t_{i,k}) - \lambda\, y_{i,\mathcal{J}(t_{i,k})}(t_{i,k}), \\
\alpha_i(t_{i,k}) &= \left( A_i(t_{i,k}) A_i(t_{i,k})^{\top} + \lambda^2 I_{m_i(t_{i,k})} \right)^{-1} r_i(t_{i,k}).
\end{aligned}
\]
Since ${A_i’}^\dagger(t_{i,k}) = A_i’(t_{i,k})^\top \left( A_i’(t_{i,k}) A_i’(t_{i,k})^\top \right)^{-1}$ and $A_i’(t_{i,k}) = \left( A_i(t_{i,k}),\ \lambda E_{\mathcal{J}(t_{i,k})} \right)$, the structure of the augmentation implies that each update involves only the components of $y_i$ indexed by $\mathbf{m}_i$.

In other words, even if all agents synchronize and aggregate their $y_i$ components during communication—as is done for $w_i(t_{i,k})$—the actual computation at each agent only requires the local subset $y_{i,\mathcal{J}(t_{i,k})}$. Therefore, $y_i$ need not be communicated across agents. In practice, the portions of $y_i$ outside of $\mathbf{m}_i$ are never involved in agent $i$’s updates. Consequently, agent $i$ only needs to maintain a local vector $y_i$ of length $m_i$, storing the relevant subset of the global variable $y$.

Based on the above discussion, Equation~\eqref{eq:augmented-kaczmarz} can be simplified into the following explicit iteration involving only $x_i(t_{i,k})$ and $y_i(t_{i,k})$:
\begin{equation}
\begin{aligned}
&r_i(t_{i,k}) = b_i(t_{i,k}) - A_i(t_{i,k}) w_i(t_{i,k}) - \lambda y_{i,\mathcal{J}(t_{i,k})}(t_{i,k}), \\
&\alpha_i(t_{i,k}) = \left( A_i(t_{i,k}) A_i(t_{i,k})^\top + \lambda^2 I_{m_i(t_{i,k})} \right)^{-1} r_i(t_{i,k}), \\
&x_i(t_{i,k+1}) = w_i(t_{i,k}) + A_i(t_{i,k})^\top \alpha_i(t_{i,k}), \\
&y_{i,\mathcal{J}(t_{i,k})}(t_{i,k+1}) = y_{i,\mathcal{J}(t_{i,k})}(t_{i,k}) + \lambda\alpha_i(t_{i,k}).
\end{aligned}
\label{eq:inconsistent-update-explicit}
\end{equation}

If the matrices $\left( A_i(t_{i,k}) A_i(t_{i,k})^\top + \lambda^2 I_{m_i(t_{i,k})} \right)^{-1}$ are precomputed and stored, the per-iteration computational complexity remains $O(m_i(t_{i,k}) n)$, without incurring any additional communication overhead.

Since the augmented system in~\eqref{eq:augmented-system} is consistent, the convergence guarantees established for~\eqref{eq:consistent-update} can directly extend to the explicit form in ~\eqref{eq:inconsistent-update-explicit}. The convergence result and a quantitative error bound for~\eqref{eq:inconsistent-update-explicit} will be given in  \textbf{Theorem 3}, where the approximation gap between the solution $\tilde{x}^*$ obtained from the augmented system and the true least-squares solution $x^*$  is discussed. 
\subsection{The Communication Strategy}

We now proceeds to discuss the communication challenges in distributed algorithms. Within distributed algorithmic frameworks, inter-agent communication is critical to ensure convergence toward the solution of the target equation.In synchronous parallel algorithms, agents must wait  until all communication is completed before proceeding to the next iteration. As the system scales, communication delays often exceed the computation time per agent, signiffcantly hindering overall efffciency. 

Asynchronous distributed algorithms break this limitation by allowing agents to proceed independently.However, empirical studies reveal that large-scale asynchronous implementations still necessitate substantial information exchange for convergence. While intensive communication guarantees algorithmic convergence, excessive communication overhead imposes significant system burdens, adversely affecting computational performance while creating unnecessary resource expenditure. For instance, Reference \cite{Liu2018Asynchronous} employs a strategy where agents broadcast updates to neighbors after every iteration. Yet such dense communication often fail to be utilized effectively by neighboring agents.

This inefficiency becomes more pronounced when communication lacks a proper control mechanism. During iteration, agent $i$ may receive multiple updates from agent $j$. As illustrated in Fig.~\ref{fig:1}, if each iteration instance of an agent is represented as a node, then node $\mathbf{A}$ aggregates information from others and itself and forwards the result to subsequent nodes $\mathbf{B}$ and $\mathbf{C}$. Node $\mathbf{B}$ further integrates data from both $\mathbf{A}$ and itself, then sends the result to $\mathbf{C}$. Clearly, the update from $\mathbf{B} \to \mathbf{C}$ carries richer information than that from $\mathbf{A} \to \mathbf{C}$, rendering the latter redundant.

Without a carefully designed communication strategy, such redundant updates may consume significant bandwidth and processing capacity, thereby degrading computational efficiency. This issue is especially prominent in large-scale distributed systems, where an efficient communication rule is essential to mitigate unnecessary transmission and reduce system load.

\begin{figure}[H]
\centering
\includegraphics[width=0.3\textwidth, height=6cm]{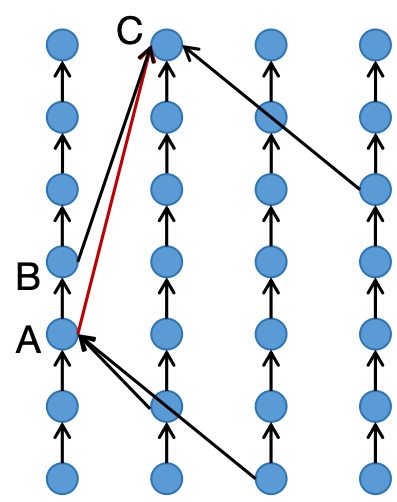}
\caption{Illustration of inter-node communication relationships.}
\label{fig:1}
\end{figure}

Building on this observation, we propose a communication strategy based on a predefined event set. Let $\mathcal{N}(t)$ denote the set of agents that initiate communication at time $t$, and let $T_k$ be the $k$-th time at which $\mathcal{N}(t) \neq \emptyset$. We define the set of communication event times as $\mathcal{T}_S = \{ T_k \}_{k=1}^{+\infty}$, where communication is triggered at each $T_k$. By designing the structure of $\mathcal{T}_S$, we achieve effective regulation of the communication process.

This event-triggered communication mechanism can be analyzed using techniques from~\cite{Liu2018Asynchronous}. We first introduce the following assumption:

\begin{assumption} There exist positive constants $\overline{T}_i$ and $T_i$ such that for all $k \geq 1$ and for all inter-agent communication delays $\Delta t$, the following inequalities hold:
\begin{equation}
\begin{aligned}
\overline{T}_i \geq &\ t_{i,k} - t_{i,k-1} \geq T_i, \\
&\Delta t \leq \Delta T.
\end{aligned}
\label{eq:communication-assumption}
\end{equation}
\end{assumption}

This assumption is reasonable: the lower bound $T_i$ reflects the minimal update frequency that each agent can achieve in practice, while the upper bound $\overline{T}_i$ guarantees that no agent becomes indefinitely inactive. From ~\eqref{eq:communication-assumption}, it follows that all communication delays $\delta_{j,i}(t_{i,k})$ are uniformly bounded.

Let $\overline{T} = \max\limits_i \left\{ \overline{T}_i \right\}$. If agent $i$ receives information from agent $j$ at time $\tau$, the delay can be expressed as:
\[\delta_{j,i}(t_{i,k}) = t_{i,k} - \tau + \tau - t_{j,h_{j,i}^k}.\]
Since $\tau \in [t_{i,k-1}, t_{i,k})$, it follows that $t_{i,k} - \tau \leq \overline{T}$. Furthermore, by \textbf{Assumption 1}, we have $\tau - t_{j,h_{j,i}^k} \leq \Delta T$. Hence, we obtain the following uniform bound on communication delay:
\[\delta_{j,i}(t_{i,k}) \leq \Delta T + \overline{T}.\]

Based on \textbf{Assumption 1}, a basic design principle for constructing the event-triggered communication set $\mathcal{T}_S$ can be obtained.

\begin{theorem} Let $\Delta T_{\text{broadcast}} = \min\limits_{k \geq 1} (T_k - T_{k-1})$ denote the minimal time interval between consecutive communication events. If
\begin{equation}
\Delta T_{\text{broadcast}} \geq 2\Delta T + \overline{T},
\label{eq:communication-design}
\end{equation}
then the redundant communication scenario described in Fig.~\ref{fig:1} will not occur.
\end{theorem}

\begin{proof}
Assume that a system-wide broadcast is initiated at time $T_s$. Let $\tau$ be the time when agent $j$ receives the broadcast and subsequently transmits its local information to agent $i$. The delay from $T_s$ to the time $t_{i,k}$ when agent $i$ utilizes the received message can be expressed as
\[\Delta t = t_{i,k} - \tau + \tau - T_s.\]
Since $\tau \in \left[t_{j,h_{j,i}^k},\ t_{j,h_{j,i}^k+1}\right)$, it follows that $$t_{i,k} - \tau \leq t_{i,k} - t_{j,h_{j,i}^k} \leq \overline{T} + \Delta T.$$ Meanwhile, from \textbf{Assumption 1}, we have $\tau - T_s \leq \Delta T$. Therefore,
\[\Delta t \leq \overline{T} + 2\Delta T \leq T_{s+1} - T_s.\]
This implies that the entire communication cascade initiated at $T_s$ concludes before the next broadcast event at $T_{s+1}$. In other words, communication processes triggered by distinct events in $\mathcal{T}_S$ are temporally isolated, eliminating logical interference or data conflicts between successive communication rounds.
\end{proof}

\textbf{Theorem 1} demonstrates that if the communication event sequence $\mathcal{T}_S = \{T_k\}_{k=1}^{+\infty}$ satisfies the condition in~\eqref{eq:communication-design}, the redundant update scenario shown in Fig.~\ref{fig:1} can be avoided. This ensures a clear and efficient communication process, minimizing resource waste. However, if $\Delta T_{\text{broadcast}}$ becomes too large, insufficient communication may slow down algorithmic convergence. This implies that the choice of $\Delta T_{\text{broadcast}}$ requires careful tuning to balance communication efficiency and convergence speed.

\section{Main Results}

In this section, we present the convergence results for the update rules in~\eqref{eq:consistent-update} and~\eqref{eq:inconsistent-update-explicit}. We begin by introducing some key definitions and analytical tools.

\subsection{Communication Graph}

Let $\mathcal{T}_i = \{ t_{i,k} \}_{k=1}^{+\infty}$ denote the iteration time sequence for agent $i$, and define the global time set as $\mathcal{T} = \bigcup\limits_{i=1}^{N} \mathcal{T}_i$. By sorting all elements in $\mathcal{T}$ in ascending order, we obtain a global time axis.

For any time $\tau \in \mathcal{T}$, we define the communication graph at time $\tau$ as $\mathcal{G}(\tau) = \left( \mathcal{N}, \mathcal{V}(\tau) \right)$, where the node set is $\mathcal{N} = \{1, 2, \dots, N\}$, and the edge set is $\mathcal{V}(\tau) = \{ (j, i) \mid j \in \mathcal{N}_i(\tau) \}$. By definition, each node in $\mathcal{G}(\tau)$ includes a self-loop.

We denote the set of communication graphs over time $\mathcal{T}$ as $\mathbf{G} = { \mathcal{G}(\tau) }_{\tau \in \mathcal{T}}$. Different choices of the event-triggered set $\mathcal{T}_S$ will result in different sequences of communication graphs $\mathbf{G}$. A central question is: under what conditions on $\mathbf{G}$ can we ensure algorithm convergence?

To account for the cumulative influence of communication over time, we introduce a graph composition operation. Let $\mathcal{A}(\mathcal{G}(\tau))$ denote the adjacency matrix of $\mathcal{G}(\tau)$. For any two graphs $\mathcal{G}_1$ and $\mathcal{G}_2$ with adjacency matrices $\mathcal{A}(\mathcal{G}_1)$ and $\mathcal{A}(\mathcal{G}_2)$, respectively, their composition is defined as:
\[\mathcal{G}_1 \circ \mathcal{G}_2 \triangleq \text{a graph with adjacency matrix } \mathcal{A}(\mathcal{G}_1)\mathcal{A}(\mathcal{G}_2).\]
This graph product will be used in subsequent analysis to characterize the connectivity of the system over consecutive time intervals.

\subsection{Persistent Communication Events}

In this part, we theoretically establish that as long as the communication graph sequence $\mathbf{G}$—generated under the designed communication scheme—satisfies a joint strong connectivity condition, the algorithm will converge regardless of the specific choice of the event-triggering sequence $\mathcal{T}_S$.

To formalize this condition, we define two events.
\begin{equation}
\begin{aligned}
\mathcal{C}(l) = \Big\{ \text{For all } \tau_1 \in \mathcal{T}, & \text{ and } \Delta l \geq l, \text{ the composed graph } \mathcal{G}(\tau_{\Delta l}) \circ \cdots \circ \mathcal{G}(\tau_2) \circ \mathcal{G}(\tau_1) \\
& \text{ is a strongly connected directed graph} \Big\},
\end{aligned}
\label{eq:connectivity-event}
\end{equation}
and the complementary event:
\begin{equation}
\begin{aligned}
\mathcal{C}(\infty) = \Big\{ \text{For all } \tau_1 \in \mathcal{T},& \text{and}\  l \geq 1, \text{the composed graph } \mathcal{G}(\tau_{\Delta l}) \circ \cdots \circ \mathcal{G}(\tau_2) \circ \mathcal{G}(\tau_1) \\
& \text{is not a strongly connected directed graph} \Big\},
\end{aligned}
\label{eq:communication-failure}
\end{equation}
where $l \geq 1$ and $\tau_1, \dots, \tau_{\Delta l}$ are consecutive time points in $\mathcal{T}$.

The event $\mathcal{C}(\infty)$ indicates a pathological scenario in which persistent communication failure prevents certain agents from ever exchanging information with others. This extreme case corresponds to complete communication breakdown across the system.

\subsection{Convergence Results}

We now establish that the iteration~\eqref{eq:consistent-update} converges under a mild connectivity condition, regardless of the structure of the matrix $A$.

\begin{theorem}
    If the probability of event $\mathcal{C}(\infty)$ occurring is zero, then the algorithm defined by~\eqref{eq:consistent-update} converges almost surely to a solution of the linear system at an exponential rate.
\end{theorem}

\textbf{Theorem 2} states that as long as the communication network retains sufficient long-term connectivity—i.e., the event $\mathcal{C}(\infty)$ does not occur—the algorithm guarantees exponential convergence. Conversely, persistent communication failure will prevent convergence.

We next present the convergence result for the inconsistent case governed by~\eqref{eq:inconsistent-update-explicit}.

\begin{theorem}
    If the probability of event $\mathcal{C}(\infty)$ occurring is zero, then the algorithm~\eqref{eq:inconsistent-update-explicit} converges almost surely at an exponential rate to an approximation $\tilde{x}^*$ of the least-squares solution $x^*$ of the inconsistent system $Ax = b$. Let $\sigma_{\min}$ denote the smallest singular value of $A$, then the relative error satisfies:
    \[\frac{\|x^* - \tilde{x}^*\|_2}{\|x^*\|_2} \leq \frac{1}{\left( \frac{\sigma_{\min}}{\lambda} \right)^2 + 1}.\]
\end{theorem}

The proofs of \textbf{Theorem 2} and \textbf{Theorem 3} are provided in the Appendix.

\section{Experiments}
\subsection{Experimental Design}

To evaluate the practical effectiveness of the proposed \emph{asynchronous distributed projection algorithm}, we first outline a unified \textbf{experimental design strategy} prior to presenting detailed results. All experiments follow a \textbf{controlled variable methodology}: we first define a set of \emph{baseline configurations}, and then perform comparative analysis along four distinct and non-overlapping dimensions:

\textbf{(i) Scalability:} We examine convergence performance and communication cost as the number of agents $\mathrm{N}$ increases under varying problem sizes $(m, n)$, to assess how the algorithm scales with system parallelism.

\textbf{(ii) Communication–Computation Trade-off:} With fixed problem size and number of agents, we vary the number of neighbors $N_n$ and the communication triggering interval $\Delta t$ to explore the trade-off between local computation and communication overhead.

\textbf{(iii) Robustness:} We inject node failure disturbances into a fixed configuration to evaluate the fault tolerance of the algorithm and quantify the performance degradation incurred.

\textbf{(iv) $\lambda$ Parameter Sensitivity:} We solve the same problem instance under different choices of the regularization parameter $\lambda$ to investigate its impact on convergence and accuracy.

\subsection{Experimental Details}
\paragraph{Parameters}

Except for the independent variables under investigation, all other hyperparameters are held constant across experiments. The notations and descriptions of all parameters are summarized in Table~\ref{tab:baseline}. Each agent utilizes a non-blocking publish–subscribe mechanism implemented via \textbf{ROS2} Topics. A unified convergence criterion is adopted: \emph{“terminate globally when any node satisfies the tolerance condition.”}

Each experiment is repeated five times. The following core performance metrics are recorded and averaged to mitigate the variance introduced by asynchronous execution: number of iterations ($\bar{k}$), computation time ($\bar{t}_{\mathrm{cmp}}$), total number of communications ($\bar{c}$), cumulative communication time ($\bar{t}_{\mathrm{comm}}$), and overall wall-clock time ($T$).

\begin{table}[H]
\centering
\small
\caption{Baseline Experimental Parameters and Evaluation Metrics}
\label{tab:baseline}
\begin{tabular}{c|lll}
\toprule
~ & \multicolumn{1}{c}{\textbf{Category}} & \multicolumn{1}{c}{\textbf{Symbol}} & \multicolumn{1}{c}{\textbf{Description}} \\
\midrule
~ & Sparsity          & $\mathrm{density}$     & Ratio of nonzero elements in matrix $A$ (set to $0.05$) \\
~ & Projection block size & $\mathrm{N}_{\mathrm{R}}$  & Number of rows used per projection step (set to $50$) \\
~ & Error tolerance   & $\mathrm{tol}$         & Convergence threshold $|x_i - x^\star| \le \mathrm{tol}$ (set to $10^{-3}$) \\
~ & Max iteration cap & $\bar{k}_{\mathrm{max}}$ & Used in~\eqref{eq:inconsistent-update-explicit} for maximum iterations per node (set to $5000$) \\
\textbf{Experimental} & Number of nodes   & $\mathrm{N}$           & Independent variable, adjusted based on problem size \\
\textbf{Parameters} & Number of neighbors & $\mathrm{N}_{\mathrm{n}}$   & Number of adjacent communication peers (default equals node count) \\
~ & Communication interval & $\Delta t$        & Integer controlling communication frequency per unit time (default: $25$) \\
~ & Failure ratio     & $\rho$                 & Proportion of failed nodes (default: $0$) \\
~ & Failure duration  & $\xi$                  & Max downtime per node in seconds (default: $0$) \\
~ & Inconsistency noise & $\sigma$             & Artificial noise added for~\eqref{eq:inconsistent-update-explicit}-related experiments (set to $1$) \\
\midrule
~ & Avg. iteration count & $\bar{k}_{\mathrm{iter}}$ & Mean number of iterations to reach tolerance \\
~ & Avg. computation time & $\bar{t}_{\mathrm{cmp}}$ & Mean total computation time per node (sec) \\
~ & Avg. communication count & $\bar{c}$        & Average number of topic-based message transmissions \\
\textbf{Evaluation} & Avg. communication time & $\bar{t}_{\mathrm{comm}}$ & Total accumulated time spent in communication (sec) \\
\textbf{Metrics} & Wall-clock time    & $T$                    & Total elapsed system runtime (sec) \\
~ & Avg. failure time  & $\bar{t}_{\mathrm{stop}}$ & Average total downtime for nodes with simulated failures (sec) \\
~ & Avg. failure count & $\bar{k}_{\mathrm{stop}}$ & Average number of failure events per affected node \\
~ & Termination error  & $\mathrm{e}_{\mathrm{stop}}$ & For~\eqref{eq:inconsistent-update-explicit}, minimum 2-norm error at convergence \\
\bottomrule
\end{tabular}
\end{table}

\paragraph{Data Generation and Preprocessing}

The linear systems used in the experiments are automatically generated and cached by a unified data management module, \texttt{DataManager}, ensuring data reusability and comparability across experiment groups. Given a problem size $(m, n)$, the goal is to construct a sparse linear system $Ax = b$, where the coefficient matrix $A \in \mathbb{R}^{m \times n}$ is a randomly generated sparse matrix with fixed sparsity.

The generation process is as follows: we use \texttt{scipy.sparse.random} to generate $A$ with a sparsity level of $\mathrm{density} = 0.05$, and convert it into a NumPy array for computational efficiency. A target solution vector $x^\star \in \mathbb{R}^n$ is sampled from a standard normal distribution, and the noiseless right-hand side is computed via $b = Ax^\star$. To simulate system inconsistency, a zero-mean Gaussian noise term with variance $\sigma^2$ can be added: $b = Ax^\star + \sigma \cdot \eta$, where $\eta \sim \mathcal{N}(0, I)$.

To facilitate convergence analysis and error evaluation, the minimum-norm solution $x^*$ is precomputed using the Moore–Penrose pseudoinverse via \texttt{pinv}$(A)$. The generated data are evenly partitioned across $N$ agents and saved as \texttt{A\_i.npy} and \texttt{b\_i.npy} files, which are loaded and distributed by the controller at experiment initialization.

\paragraph{Pascal-Triangle-Based Communication Topology}

To construct a well-structured and analytically tractable communication graph, we adopt a \textbf{Pascal-triangle-based priority linking strategy} to define inter-agent connectivity. Nodes are first indexed according to a Pascal triangle hierarchy, where the $r$-th row contains $r{+}1$ nodes. This results in a layered structure.

Each node is then connected to its “parent nodes” in the row above—specifically, the left parent $(r{-}1, c{-}1)$ and right parent $(r{-}1, c)$—as well as its “child nodes” in the row below: $(r{+}1, c)$ and $(r{+}1, c{+}1)$, and its immediate left and right neighbors within the same row. If the number of neighbors is still below the predefined threshold $N_n$, additional edges are added randomly to preserve connectivity while capping the maximum degree at $N_n$.

This hybrid strategy combines structural regularity with controlled randomness and offers the following benefits:
(1) Naturally constructed “parent–child” paths promote efficient state propagation;
(2) Balanced degree distribution prevents communication bottlenecks;
(3) Clear and analyzable topology facilitates visualization and theoretical interpretation.

\paragraph{Communication Events Design}

To evaluate the impact of event-triggered communication on performance, we define a simple rule for triggering communication events in the experiments. Specifically, each agent $i$ communicates with its neighbors at iteration steps satisfying:
\[\big\{t_{i,k} \ \big| \ \text{mod}(k, \Delta t) = 0, \, k > 0\big\},\]
i.e., agent $i$ broadcasts to neighbors every $\Delta t$ iterations. All other parameters are held constant across experiments, with only $\Delta t$ varying. The number of agents is fixed to a minimal configuration $\text{N}_{\min}$ to isolate the communication frequency effect.

\paragraph{Hardware and Software Environment}

All experiments were conducted on the AutoDL cloud server platform, with each node equipped with dual 64-core AMD EPYC 9654 vCPUs. The host operating system was Ubuntu 22.04 LTS. The software environment included the ROS2 Humble middleware for communication, Python 3.11, and the numerical computation libraries NumPy 2.2 and SciPy 1.15. This infrastructure provides sufficient concurrency and network bandwidth to support stable, large-scale asynchronous experiments with multiple agents.

\paragraph{Remarks on Asynchronous Randomness}

In asynchronous distributed systems, node-level computation and communication are influenced by numerous stochastic factors, including OS-level thread scheduling, network packet queuing, and the underlying NUMA memory architecture. As a result, even with identical configurations, repeated runs may yield slight variations in performance.

To mitigate bias introduced by such randomness, we adopt the following experimental protocols:
(1) The random seeds for both NumPy and ROS initializations are fixed to ensure reproducible generation of $A$, $b$, and the communication topology;
(2) Each experiment is repeated five times, and results are averaged across runs;
(3) Performance is evaluated using statistical metrics—such as iteration count, communication time, and wall-clock duration—rather than relying on single-run measurements.

Hence, the conclusions presented in this paper should be interpreted as \emph{statistical trends}, rather than exact pointwise characterizations of system performance.

\subsection{Experimental Results}

\paragraph{Effect of Varying Number of Nodes}

According to the proof of \textbf{Theorem 1}, if the projection matrix constructed from the data held by a single agent already spans the entire column space, then additional information from other agents becomes redundant, and inter-agent communication no longer contributes to error reduction. Based on this observation, we design the number of agents $\mathrm{N}$ for a given problem size $(m, n)$ as follows:
\[\text{N}(\theta_1) = \left\lceil \frac{m}{n \, \theta_1} \right\rceil, \quad \theta_1 \in \left\{0.2 + 0.075\,i \mid i = 1, \dots, 8 \right\},\]
where $\theta_1$ represents the proportion of the global data allocated to each agent. This ensures that, for any configuration, each agent holds approximately $\theta_1$ fraction of the total data. Results are summarized in Table~\ref{tab:节点数}.

\begin{table}[H]
\centering
\footnotesize
\caption{Performance Metrics under Varying Agent Counts across Different Problem Sizes}
\begin{subtable}[t]{0.48\textwidth}
\centering
\begin{tabular}{ccccccc}
\toprule
\textbf{(m,n)} & $\mathrm{N}$ & $\bar{k}_{\mathrm{iter}}$ & $\bar{t}_{\mathrm{comm}}$ & $\bar{c}$ & $\bar{t}_{\mathrm{comm}}$ & $T$ \\
\midrule
~ & \textbf{13.00} & \textbf{2,000.38} & \textbf{10.07} & \textbf{953.08} & \textbf{6.12} & \textbf{16.18} \\
~ & 14.00 & 2,065.14 & 11.96 & 1,066.21 & 7.38 & 19.34 \\
~ & 16.00 & 2,130.56 & 14.69 & 1,268.75 & 9.25 & 23.94 \\
\multirow{2}{*}{\textbf{30000$\times$3000}} & 19.00 & 2,389.84 & 17.53 & 1,697.74 & 10.93 & 28.47 \\
~ & 22.00 & 2,527.41 & 23.15 & 2,103.32 & 15.01 & 38.17 \\
~ & 27.00 & 2,837.70 & 21.02 & 2,917.78 & 10.80 & 31.82 \\
~ & 35.00 & 7,350.83 & 72.09 & 7,517.60 & 14.26 & 86.35 \\
~ & 50.00 & 12,218.90 & 174.17 & 9,740.64 & 16.75 & 190.92 \\
\midrule
~ & 13.00 & 2,624.69 & 13.35 & \textbf{1,253.54} & 8.25 & 21.60 \\
~ & 14.00 & 2,688.57 & 15.70 & 1,391.57 & 9.75 & 25.45 \\
~ & 16.00 & 2,854.81 & 19.05 & 1,701.25 & 12.74 & 31.79 \\
\multirow{2}{*}{\textbf{40000$\times$4000}} & \textbf{19.00} & \textbf{2,583.84} & \textbf{8.25} & 1,848.84 & \textbf{6.17} & \textbf{14.42} \\
~ & 22.00 & 3,339.50 & 28.66 & 2,780.77 & 20.92 & 49.58 \\
~ & 27.00 & 4,268.48 & 44.27 & 4,332.48 & 25.23 & 69.50 \\
~ & 35.00 & 3,819.83 & 21.16 & 4,646.97 & 6.72 & 27.87 \\
~ & 50.00 & 4,536.54 & 49.50 & 6,019.08 & 15.32 & 64.82 \\
\midrule
~ & 21.00 & \textbf{2,007.90} & 18.31 & \textbf{1,582.14} & 11.96 & 30.27 \\
~ & 24.00 & 2,220.75 & 21.06 & 2,016.38 & 13.32 & 34.38 \\
~ & 27.00 & 2,606.63 & 27.12 & 2,660.07 & 13.91 & 41.03 \\
\multirow{2}{*}{\textbf{50000$\times$3000}} & 31.00 & 5,869.65 & 62.42 & 4,476.39 & 14.47 & 76.89 \\
~ & 37.00 & 14,556.59 & 169.43 & 10,034.11 & 21.84 & 191.27 \\
~ & \textbf{45.00} & 2,365.38 & \textbf{8.26} & 4,095.02 & \textbf{4.14} & \textbf{12.40} \\
~ & 59.00 & 3,732.63 & 29.20 & 6,061.05 & 5.12 & 34.32 \\
~ & 84.00 & 4,633.73 & 38.15 & 12,195.23 & 10.33 & 48.48 \\
\bottomrule
\end{tabular}
\caption{Set A}
\end{subtable}
\hfill
\begin{subtable}[t]{0.48\textwidth}
\centering
\begin{tabular}{ccccccc}
\toprule
\textbf{(m,n)} & $\mathrm{N}$ & $\bar{k}_{\mathrm{iter}}$ & $\bar{t}_{\mathrm{comm}}$ & $\bar{c}$ & $\bar{t}_{\mathrm{comm}}$ & $T$ \\
\midrule
~ & 13.00 & 2,979.62 & 11.79 & \textbf{1,423.15} & 7.19 & 18.99 \\
~ & 14.00 & 3,321.71 & 22.29 & 1,720.50 & 14.47 & 36.76 \\
~ & \textbf{16.00} & \textbf{2,968.69} & \textbf{5.79} & 1,772.06 & \textbf{4.50} & \textbf{10.28} \\
\multirow{2}{*}{\textbf{50000$\times$5000}} & 19.00 & 3,656.84 & 29.44 & 2,610.05 & 21.82 & 51.26 \\
~ & 22.00 & 4,122.23 & 38.41 & 3,444.77 & 28.89 & 67.30 \\
~ & 27.00 & 4,635.41 & 49.45 & 4,764.96 & 32.65 & 82.10 \\
~ & 35.00 & 4,962.66 & 48.23 & 4,486.71 & 20.54 & 68.77 \\
~ & 50.00 & 5,402.64 & 28.76 & 9,136.68 & 10.83 & 39.59 \\
\midrule
~ & 19.00 & \textbf{2,526.58} & \textbf{15.12} & \textbf{1,786.74} & \textbf{9.93} & 25.05 \\
~ & 21.00 & 2,786.00 & 24.46 & 2,206.95 & 15.51 & 39.97 \\
~ & 24.00 & 2,867.38 & 28.79 & 2,612.92 & 19.28 & 48.08 \\
\multirow{2}{*}{\textbf{60000$\times$4000}} & 28.00 & 3,550.32 & 39.58 & 3,761.75 & 20.72 & 60.30 \\
~ & 33.00 & 6,191.21 & 58.53 & 4,416.88 & 10.13 & 68.66 \\
~ & 41.00 & 6,846.66 & 78.69 & 7,352.20 & 12.68 & 91.37 \\
~ & \textbf{53.00} & 3,374.17 & 18.05 & 6,438.79 & 6.65 & \textbf{24.70} \\
~ & 75.00 & 3,915.13 & 19.33 & 11,539.15 & 9.97 & 29.30 \\
\midrule
~ & 20.00 & 3,183.05 & 28.68 & \textbf{2,405.35} & 20.99 & 49.68 \\
~ & 23.00 & 3,407.04 & 33.36 & 2,961.30 & 23.09 & 56.46 \\
~ & \textbf{26.00} & \textbf{2,912.85} & \textbf{12.50} & 2,895.65 & \textbf{9.21} & \textbf{21.70} \\
\multirow{2}{*}{\textbf{80000$\times$5000}} & 30.00 & 5,731.30 & 65.90 & 5,736.33 & 29.65 & 95.54 \\
~ & 35.00 & 4,446.34 & 30.58 & 4,778.37 & 8.47 & 39.05 \\
~ & 44.00 & 6,064.98 & 60.31 & 7,230.57 & 11.78 & 72.09 \\
~ & 57.00 & 4,093.14 & 13.11 & 9,128.95 & 9.40 & 22.51 \\
~ & 80.00 & 5,357.69 & 19.90 & 16,856.49 & 16.06 & 35.96 \\
\bottomrule
\end{tabular}
\caption{Set B}
\end{subtable}
\label{tab:节点数}
\end{table}

\textbf{Sensitivity to agent Count:} Across six problem configurations, we observe a highly consistent trend when the per-agent data ratio $\theta_1$ is uniformly decreased from $0.80$ to $0.20$ in eight steps:
\begin{enumerate}
    \item[(i)] When $\theta_1 \ge 0.50$ (i.e., a small number of agents, each holding sufficient data to span the column space), the average number of iterations and communication rounds increase only marginally. Wall-clock time $T$ remains low, indicating a balanced trade-off between \emph{computational speedup} and \emph{communication load}.
    \item[(ii)] As $\theta_1$ drops to the range of $[0.30, 0.35]$, local projections become insufficient to independently span the column space, resulting in a sharp increase in dependence on inter-agent information. Consequently, both iteration count and communication overhead increase significantly.
    \item[(iii)] When $\theta_1$ approaches $0.20$ (i.e., maximum $\mathrm{N}$), the “information deficiency” effect is further amplified by network queuing and ROS2 topic contention. For larger problem sizes (e.g., $30{,}000 \times 3{,}000$ and $50{,}000 \times 3{,}000$), this leads to a second surge in iterations, forming a characteristic “U-shaped” curve in wall-clock time.
\end{enumerate}

Interestingly, systems with higher row-to-column ratios exhibit greater tolerance to small $\theta_1$, and their optimal $\theta_1$ shifts slightly rightward. This suggests that in “row-rich” settings, individual agents still retain substantial informative content.

\textbf{In summary, maintaining a per-agent data ratio of $\theta_1 \gtrsim 0.30$ effectively suppresses both iteration explosion and communication congestion across all tested scales.} It should be noted that due to unavoidable asynchronous effects such as ROS2 scheduling, network queuing, and NUMA contention, the results reflect \emph{macroscopic trends} rather than precise performance benchmarks. Nonetheless, these findings offer practical guidance for estimating the \emph{minimum feasible node count} and the \emph{upper bound of parallel efficiency} in real-world deployments.

\paragraph{Neighbor Count}

In this section, we evaluate the sensitivity of algorithm performance to the number of neighbors per agent. The experiment setup is based on the same linear systems used in the \textbf{Varying Agent Count} experiments. For each problem size, we fix the number of agents $\mathrm{N}$ to the optimal value $\mathrm{N}_{\min}$ previously identified as minimizing wall-clock time $T$ in~\ref{tab:节点数}. All other parameters are kept consistent, except for the neighbor count $\mathrm{N}_{\text{n}}$, which is selected from the set:
\[
\Big\{\mathrm{N}_{\text{n}} \;\Big|\; \mathrm{N}_{\text{n}} = \big\lceil \mathrm{N} \cdot \theta_2 \big\rceil,\ \theta_2 \in \big\{0.1 \times i \;\big|\; i = 1, \dots, 8\big\} \Big\}.
\]
Results are summarized in Table~\ref{tab:邻居数}.

\begin{table}[H]
\centering
\footnotesize
\caption{Impact of Neighbor Count Across Different Problem Scales}
\begin{subtable}[t]{0.48\textwidth}
\centering
\begin{tabular}{ccccccc}
\toprule
\textbf{(m,n)} & $\mathrm{N}_{\mathrm{n}}$ & $\bar{k}_{\mathrm{iter}}$ & $\bar{t}_{\mathrm{comm}}$ & $\bar{c}$ & $\bar{t}_{\mathrm{comm}}$ & $T$ \\
\midrule
~ & 2 & 2486.31 & 8.99 & \textbf{196.85} & 1.24 & 10.23 \\
~ & 3 & 1939.92 &\textbf{ 5.80} & 223.85 & \textbf{1.18} & \textbf{6.98} \\
~ & 4 & 1809.69 & 8.23 & 286.54 & 2.03 & 10.26 \\
\textbf{30000$\times$3000} & 6 & \textbf{1801.54} & 6.66 & 417.15 & 2.45 & 9.11 \\
$\mathrm{N}=13$ & 7 & 1783.85 & 6.56 & 488.54 & 2.61 & 9.17 \\
~ & 8 & 1863.69 & 8.93 & 565.54 & 3.72 & 12.65 \\
~ & 10 & 1953.92 & 9.15 & 749.08 & 4.62 & 13.77 \\
~ & 11 & 1930.85 & 7.83 & 825.62 & 4.56 & 12.39 \\
\midrule
~ & 2 & 4258.16 & 24.11 & \textbf{319.89} & \textbf{3.15} & 27.26 \\
~ & 4 & 2601.89 & 16.62 & 413.79 & 3.90 & \textbf{20.52} \\
~ & 6 & \textbf{2514.16 }& \textbf{16.06} & 588.11 & 5.56 & 21.62 \\
\textbf{40000$\times$4000} & 8 & 2574.89 & 18.12 & 806.74 & 7.19 & 25.31 \\
$\mathrm{N}=19$ & 10 & 2613.05 & 20.07 & 1025.95 & 9.67 & 29.74 \\
~ & 12 & 2705.26 & 20.55 & 1276.32 & 11.82 & 32.37 \\
~ & 14 & 2822.05 & 22.69 & 1544.58 & 13.67 & 36.36 \\
~ & 16 & 3007.68 & 21.70 & 1830.79 & 13.40 & 35.10 \\
\midrule
~ & 5 & \textbf{2121.20} & 26.10 & \textbf{417.24} & 6.36 & 32.46 \\
~ & 9 & 3536.71 & 33.11 & 643.33 & 3.29 & 36.40 \\
~ & 14 & 4959.89 & 42.80 & 978.36 & \textbf{2.46} & 45.26 \\
\textbf{50000$\times$3000} & 18 & 9301.33 & 109.88 & 2919.16 & 14.75 & 124.63 \\
$\mathrm{N}=45$ & 23 & 2205.56 & \textbf{12.81} & 1865.49 & 3.32 & \textbf{16.13} \\
~ & 27 & 4666.31 & 53.15 & 2074.11 & 7.70 & 60.85 \\
~ & 32 & 7461.38 & 87.68 & 5634.82 & 11.64 & 99.32 \\
~ & 36 & 2684.33 & 16.64 & 3477.44 & 4.02 & 20.66 \\
\bottomrule
\end{tabular}
\caption{Set A}
\end{subtable}
\hfill
\begin{subtable}[t]{0.48\textwidth}
\centering
\begin{tabular}{ccccccc}
\toprule
\textbf{(m,n)} & $\mathrm{N}_{\mathrm{n}}$ & $\bar{k}_{\mathrm{iter}}$ & $\bar{t}_{\mathrm{comm}}$ & $\bar{c}$ & $\bar{t}_{\mathrm{comm}}$ & $T$ \\
\midrule
~ & 2 & 3844.44 & 21.61 & \textbf{306.00} & \textbf{3.05} & 24.66 \\
~ & 4 & 2940.38 & 16.18 & 452.12 & 4.17 & \textbf{20.35} \\
~ & 5 & 2904.75 & \textbf{15.98} & 562.62 & 5.19 & 21.17 \\
\textbf{50000$\times$5000} & 7 & 2971.31 & 16.40 & 826.88 & 7.27 & 23.67 \\
$\mathrm{N}=16$  & 8 & 2990.25 & 17.79 & 918.94 & 7.98 & 25.77 \\
~ & 10 & 3017.88 & 18.69 & 1183.62 & 10.61 & 29.30 \\
~ & 12 & 3181.81 & 20.73 & 1502.12 & 13.14 & 33.87 \\
~ & 13 & 3325.88 & 27.09 & 1685.62 & 16.52 & 43.61 \\
\midrule
~ & 6 & \textbf{2997.96} & 41.74 & \textbf{710.60} & 12.60 & 54.34 \\
~ & 11 & 4707.23 & 66.63 & 1561.25 & 16.40 & 83.03 \\
~ & 16 & 4151.70 & \textbf{37.43} & 2505.45 & \textbf{8.57} & \textbf{46.00} \\
\textbf{60000$\times$4000} & 22 & 9057.58 & 130.67 & 3536.96 & 13.16 & 143.83 \\
$\mathrm{N}=53$ & 27 & 7817.55 & 100.66 & 6699.91 & 18.83 & 119.49 \\
~ & 32 & 8517.45 & 130.14 & 5977.43 & 19.10 & 149.24 \\
~ & 38 & 19740.64 & 305.04 & 6433.70 & 22.64 & 327.68 \\
~ & 43 & 11700.68 & 169.56 & 7009.15 & 15.86 & 185.42 \\
\midrule
~ & 3 & 2947.04 & 26.17 & \textbf{351.00} & \textbf{4.88} & 31.05 \\
~ & 6 & \textbf{2616.38} & 26.61 & 605.42 & 8.66 & 35.27 \\
~ & 8 & 2647.81 & 24.43 & 840.81 & 10.25 & 34.68 \\
\textbf{80000$\times$5000} & 11 & 2825.27 & 27.80 & 1223.85 & 14.03 & 41.83 \\
$\mathrm{N}=26$ & 13 & 2879.31 & 29.83 & 1481.50 & 16.71 & 46.54 \\
~ & 16 & 2621.23 & \textbf{11.21} & 1641.31 & 7.89 & \textbf{19.10} \\
~ & 19 & 3342.65 & 35.44 & 2481.08 & 23.55 & 58.99 \\
~ & 21 & 3260.31 & 30.22 & 2678.23 & 21.19 & 51.41 \\
\bottomrule
\end{tabular}
\caption{Set B}
\end{subtable}
\label{tab:邻居数}
\end{table}

\textbf{Neighbor Count Sensitivity Analysis:} Under fixed agent count $\mathrm{N}_{\min}$, we vary the maximum allowable neighbor count $\mathrm{N}_{\text{n}}$ and observe the following:

\begin{enumerate}
    \item[(i)] \textbf{Convergence Efficiency:} As $\mathrm{N}_{\text{n}}$ increases, the average iteration count $\bar{k}$ exhibits a three-phase trend: initially decreasing due to enhanced information diffusion, then plateauing, and slightly rising afterward due to communication overhead and queuing delays outweighing synchronization benefits.
    \item[(ii)] \textbf{Communication Time:} The total communication count $\bar{c}$ increases nearly linearly with $\mathrm{N}_{\text{n}}$. However, due to ROS2’s internal batching for large messages, the growth in total communication time $\bar{t}_{\mathrm{comm}}$ is often sublinear with respect to $\bar{c}$.
    \item[(iii)] \textbf{Wall-Clock Time:} Consequently, $T$ forms a classic U-shaped curve. Across all problem scales, the minimum $T$ is consistently achieved when $\mathrm{N}_{\text{n}} \approx [0.2,0.35]\mathrm{N}$ (e.g., for $30000 \times 3000$, $T_{\min} = 6.98$s at $\mathrm{N}_{\text{n}} = 3$). Larger problem sizes or higher row-to-column ratios tend to shift the optimal neighbor ratio slightly lower, indicating that communication bottlenecks are more pronounced in larger-scale systems. Extreme points, such as the case $60000 \times 4000$ with $\mathrm{N}_{\text{n}} = 38$, show sharp increases in iteration count due to additional queuing overhead triggered by ROS2 topic buffer saturation.
\end{enumerate}

\textbf{In summary, a moderate neighborhood size of approximately $0.25\mathrm{N}$ strikes a favorable balance between information diffusion and communication cost, consistently achieving near-optimal performance across problem scales.} This insight implies that, particularly in resource-constrained scenarios, tuning neighbor degree offers a practical alternative to scaling agent count or expanding global communication scope, thus effectively minimizing overall wall-clock time without compromising convergence.

\paragraph{Communication Frequency}

In this experiment, we investigate how the frequency of communication events influences algorithmic performance by varying the communication triggering interval $\Delta t$ as follows:
\[\Delta t \in \{1, 3, 5, 7, 9, 15, 30, 50, 70, 100\}.\]
Guided by the principle of parsimony, we refrain from designing overly complex communication triggers. This range of $\Delta t$ values is sufficient to reveal the relationship between communication frequency and convergence behavior. Results are summarized in Table~\ref{tab:intervals}.

\begin{table}[H]
\centering
\footnotesize
\caption{Performance Metrics under Different Communication Intervals across Problem Sizes}
\begin{subtable}[t]{0.48\textwidth}
\centering
\begin{tabular}{ccccccc}
\toprule
\textbf{(m,n)} & $\Delta t$ & $\bar{k}_{\mathrm{iter}}$ & $\bar{t}_{\mathrm{cmp}}$ & $\bar{c}$ & $\bar{t}_{\mathrm{comm}}$ & $T$ \\
\midrule
~ & 1 & 3714.69 & 49.85 & 13513.54 & 51.36 & 101.21 \\
~ & 3 & 2703.77 & 24.78 & 7395.69 & 29.48 & 54.26 \\
~ & 5 & 2167.08 & 19.11 & 4868.85 & 23.30 & 42.41 \\
~ & 7 & 2020.92 & 14.28 & 3200.77 & 14.48 & 28.76 \\
\textbf{30000$\times$3000} & 9 & 1920.00 & 13.10 & 2436.46 & 11.75 & 24.85 \\
$\mathrm{N}=13$ & 15 & 1801.54 & 11.21 & 1415.15 & 8.30 & 19.51 \\
~ & 30 & 1644.77 & 6.32 & 651.31 & 3.51 & 9.83 \\
~ & 50 & 2161.62 & 7.23 & 511.69 & 2.73 & 9.96 \\
~ & 70 & \textbf{1580.85} & \textbf{4.69} & \textbf{264.00} & \textbf{1.64} & \textbf{6.33} \\
~ & 100 & 2522.38 & 5.72 & 294.85 & 1.74 & 7.46 \\
\midrule
~ & 1 & 5506.89 & 116.21 & 20213.37 & 79.80 & 196.01 \\
~ & 3 & 4520.95 & 42.74 & 21555.37 & 37.47 & 80.21 \\
~ & 5 & 3371.58 & 15.79 & 11671.05 & 18.49 & 34.28 \\
~ & 7 & 3284.00 & 21.82 & 8180.26 & 20.71 & 42.53 \\
\textbf{40000$\times$4000} & 9 & 3185.68 & 34.57 & 6061.47 & 33.19 & 67.76 \\
$\mathrm{N}=19$ & 15 & 2741.53 & 24.19 & 3235.47 & 21.28 & 45.47 \\
~ & 30 & 2625.05 & 17.98 & 1549.68 & 11.69 & 29.67 \\
~ & 50 & 3329.42 & 19.07 & 1187.26 & 9.72 & 28.79 \\
~ & 70 & \textbf{2366.53} & \textbf{13.28} & \textbf{597.11} & \textbf{5.40} & \textbf{18.68} \\
~ & 100 & 3822.21 & 19.30 & 676.26 & 6.09 & 25.39 \\
\midrule
~ & 1 & 6062.09 & 48.96 & 116508.33 & 32.94 & 81.90 \\
~ & 3 & 4185.67 & 114.63 & 19897.40 & 14.56 & 129.19 \\
~ & 5 & \textbf{3066.58} & 53.30 & 16251.02 & 10.03 & 63.33 \\
~ & 7 & 7103.36 & 132.90 & 6563.36 & 14.37 & 147.27 \\
\textbf{50000$\times$3000} & 9 & 8793.76 & 152.13 & 6334.31 & 15.29 & 167.42 \\
$\mathrm{N}=45$ & 15 & 3506.82 & \textbf{30.15} & 6282.53 & \textbf{4.98} & \textbf{35.13} \\
~ & 30 & 5336.40 & 52.21 & 6049.38 & 8.04 & 60.25 \\
~ & 50 & 6913.33 & 72.92 & 3958.49 & 7.63 & 80.55 \\
~ & 70 & 8363.20 & 96.22 & 3693.44 & 11.35 & 107.57 \\
~ & 100 & 5381.36 & 58.75 & \textbf{2333.56} & 11.58 & 70.33 \\
\bottomrule
\end{tabular}
\caption{Set A}
\end{subtable}
\hfill
\begin{subtable}[t]{0.48\textwidth}
\centering
\begin{tabular}{ccccccc}
\toprule
\textbf{(m,n)} & $\Delta t$ & $\bar{k}_{\mathrm{iter}}$ & $\bar{t}_{\mathrm{cmp}}$ & $\bar{c}$ & $\bar{t}_{\mathrm{comm}}$ & $T$ \\
\midrule
~ & 1 & 6044.56 & 122.15 & 21962.06 & 103.16 & 225.31 \\
~ & 3 & 4995.69 & 64.16 & 14631.81 & 64.62 & 128.78 \\
~ & 5 & 4215.25 & 45.55 & 10946.44 & 58.25 & 103.80 \\
~ & 7 & 3536.50 & 31.98 & 7264.56 & 42.97 & 74.95 \\
\textbf{50000$\times$5000} & 9 & 3355.88 & 30.51 & 5428.81 & 35.73 & 66.24 \\
$\mathrm{N}=16$ & 15 & 3142.94 & 25.50 & 3121.62 & 25.62 & 51.12 \\
~ & 30 & 2796.06 & 11.26 & 1391.00 & 7.45 & 18.71 \\
~ & 50 & 3442.44 & 17.37 & 1023.56 & 9.49 & 26.86 \\
~ & 70 & \textbf{2661.88} & \textbf{12.70} & \textbf{563.19} & \textbf{5.06} & \textbf{17.76} \\
~ & 100 & 3932.88 & 15.64 & 581.94 & 5.41 & 21.05 \\
\midrule
~ & 1 & 8637.68 & 166.58 & 125995.15 & 52.27 & 218.85 \\
~ & 3 & 3280.25 & 46.89 & 42873.57 & 37.08 & 83.97 \\
~ & 5 & 6211.98 & 87.06 & 37079.83 & 21.88 & 108.94 \\
~ & 7 & 5313.42 & 63.79 & 23975.23 & 15.99 & 79.78 \\
\textbf{60000$\times$4000} & 9 & 3790.47 & 14.48 & 21633.25 & 12.78 & 27.26 \\
$\mathrm{N}=53$ & 15 & 5116.06 & 49.72 & 10735.79 & 10.55 & 60.27 \\
~ & 30 & \textbf{3006.96} & 7.55 & 5180.66 & \textbf{5.24} & \textbf{12.79} \\
~ & 50 & 13646.17 & 172.46 & 7578.92 & 16.51 & 188.97 \\
~ & 70 & 3365.83 & \textbf{7.07} & \textbf{2479.08} & 6.02 & 13.09 \\
~ & 100 & 6497.74 & 61.28 & 3142.98 & 7.54 & 68.82 \\
\midrule
~ & 1 & 6720.46 & 239.01 & 6907.38 & 43.64 & 282.65 \\
~ & 3 & 3167.88 & 51.13 & 12196.04 & 29.84 & 80.97 \\
~ & 5 & 3884.27 & 21.13 & 17713.65 & 19.17 & 40.30 \\
~ & 7 & 5543.04 & 65.57 & 14847.38 & 33.70 & 99.27 \\
\textbf{80000$\times$5000} & 9 & 6585.15 & 91.30 & 14110.08 & 45.20 & 136.50 \\
$\mathrm{N}=26$ & 15 & 2642.23 & 23.19 & 4346.42 & 22.18 & 45.37 \\
~ & 30 & 3092.42 & 33.73 & 2495.65 & 21.83 & 55.56 \\
~ & 50 & 3715.15 & 32.28 & 1842.69 & 16.94 & 49.22 \\
~ & 70 & \textbf{2593.65} & \textbf{19.15} & \textbf{915.96} & \textbf{8.85} & \textbf{28.00} \\
~ & 100 & 4252.77 & 33.87 & 1047.54 & 12.46 & 46.33 \\
\bottomrule
\end{tabular}
\caption{Set B}
\end{subtable}
\label{tab:intervals}
\end{table}

\textbf{Sensitivity to Communication Interval:} Keeping the number of agents and neighbor degree fixed, we vary only the communication interval $\Delta t$. The experiments consistently reveal a clear \textit{iteration–communication trade-off} across all problem scales:

\begin{enumerate}
    \item[(i)] \textbf{High-Frequency Communication} ($\Delta t \le 3$) leads to message flooding, with total communication count and time increasing substantially. ROS2 topic queues become congested, introducing queuing delays and elevating wall-clock time $T$.
    \item[(ii)] \textbf{Moderate Communication Frequency} ($\Delta t \approx 9$–$30$, depending on problem size) maintains low iteration counts while significantly reducing communication overhead, thereby minimizing $T$. For example, $30000 \times 3000$ achieves $T_{\min} = 6.33$s at $\Delta t = 70$; $40000 \times 4000$ reaches $29.67$s at $\Delta t = 30$; and $50000 \times 3000$ attains an optimal $T = 35.13$s at $\Delta t = 15$.
    \item[(iii)] \textbf{Sparse Communication} ($\Delta t \ge 70$) minimizes communication time but slows down information propagation, causing iteration counts to rebound sharply. For instance, in the $30000 \times 3000$ case, $\bar{k}$ first drops from $1644$ to $1580$ as $\Delta t$ increases, but rises again to $2522$ at $\Delta t = 100$, driving $T$ back up. Systems with higher row-to-column ratios and greater relative communication overhead tend to prefer smaller $\Delta t$, aligning with the strategy of “early convergence with less communication” under constrained node degrees.
\end{enumerate}

These findings confirm that \textbf{both overly frequent and overly sparse communication degrade overall efficiency}. Instead, there exists a \textit{problem-size-dependent optimal interval range} that balances fast convergence and minimal communication burden.

However, periodic communication remains fundamentally limited in adapting to real-time system dynamics. When local residuals stagnate or sudden network congestion occurs, periodic triggers cannot respond immediately, leading to inefficiencies. Therefore, \emph{further research on adaptive event-triggered strategies}—based on residual thresholds, bandwidth monitoring, or load feedback—is warranted. Such methods may preserve current optimal $T$ while compressing peak communication cost, enabling scalable deployment under even larger or denser topologies.

\paragraph{Robustness Evaluation}

To validate the robustness of the proposed algorithm, we design experiments simulating agent failures—an inevitable challenge in real-world distributed deployments. One key advantage of asynchronous distributed algorithms lies in their intrinsic tolerance to such partial failures, as agents operate independently without requiring global synchronization.

To emulate failure scenarios, we implement a \emph{single-parameter self-halting mechanism} within each agent. Specifically, given a positive parameter $\xi > 0$ (referred to as the \textbf{failure intensity}), each agent samples a failure arrival interval $K$ from an exponential distribution over its local iteration timeline:
\[K \sim \mathrm{Exp}(\xi), \qquad \mathbb{E}[K] = \frac{1}{\xi}.\]
After completing $K$ local iterations, the agent enters a halt state and suspends all computation for a duration sampled from an independent exponential distribution:
\[T_{\text{halt}} \sim \mathrm{Exp}(1/\xi), \qquad \mathbb{E}[T_{\text{halt}}] = \xi.\]
Once halted, the agent automatically resumes and restarts the cycle by drawing a new failure interval $K$. Thus, the parameter $\xi$ jointly governs the \emph{frequency of interruptions} (mean $1/\xi$ iterations between failures) and the \emph{severity of halts} (mean downtime $\xi$ seconds), offering a unified control of both \textbf{failure probability} and \textbf{disruption magnitude}.

In each experiment, we first fix the total number of agents $N$ and then randomly select a subset of agents to activate the failure mechanism. The number of failed nodes $N_{\text{stop}}$ is drawn from:
\[\Big\{ \mathrm{N}_{\text{stop}} \;\Big|\; \mathrm{N}_{\text{stop}} = \big\lceil \mathrm{N} \cdot \theta_3 \big\rceil,\ \theta_3 \in \big\{0.05, 0.10, \ldots, 0.40\big\} \Big\}.\]
The remaining agents operate normally. By gradually increasing both $\mathrm{N}_{\text{stop}}$ and $\xi$, we systematically observe the algorithm’s convergence characteristics and performance degradation under dual perturbations of \emph{failure ratio} and \emph{failure intensity}.

This experimental setup allows us to demonstrate the asynchronous framework’s resilience: \textbf{local agent failures do not compromise global convergence}. Final results are reported in Tables~\ref{tab:failure_ratio} and~\ref{tab:failure_duration}.

\begin{table}[!t]
\centering
\footnotesize
\caption{Performance Under Different Failure Ratios Across Problem Sizes}
\begin{subtable}[t]{0.48\textwidth}
\centering
\begin{tabular}{cccccccc}
\toprule
\textbf{(m,n)} & $\rho$ & $\bar{k}_{\mathrm{stop}}$ & $\bar{t}_{\mathrm{stop}}$ & $\bar{k}_{\mathrm{iter}}$ & $\bar{t}_{\mathrm{cmp}}$ & $\bar{c}$ & $\bar{t}_{\mathrm{comm}}$ \\
\midrule
~ & 10\% & 34 & 0.79 & 2003 & 9.51 & 915.77 & 5.69 \\
~ & 15\% & 33.5 & 0.75 & 1935.38 & 8.94 & 885.08 & 4.87 \\
~ & 20\% & 33.33 & \textbf{0.55} & 1976.31 & 9.24 & 890.85 & 4.70 \\
\textbf{30000$\times$3000} & 25\% & 28.25 & \textbf{0.55} & 1932.92 & 9.50 & 867.54 & 5.01 \\
$\mathrm{N}=13$ & 30\% & 37.25 & 0.77 & 1973.54 & 9.91 & 865.23 & 5.02 \\
~ & 35\% & 34.6 & 0.60 & 1975.46 & 9.96 & 861.92 & 4.93 \\
~ & 40\% & \textbf{29.5} & 0.60 & \textbf{1752.31} & \textbf{4.69} & \textbf{785.38} & \textbf{2.58} \\
~ & 50\% & 34.29 & 0.65 & 1962.46 & 9.85 & 817.54 & 4.94 \\
\midrule
~ & 10\% & 51 & 1.08 & 3032.68 & 23.80 & 2128.42 & 15.65 \\
~ & 15\% & 50.33 & 1.05 & 2673.79 & \textbf{12.90} & 1870.05 & \textbf{9.06} \\
~ & 20\% & 50.5 & 0.94 & 2993.58 & 23.24 & 2054.53 & 15.55 \\
\textbf{40000$\times$4000} & 25\% & \textbf{41.8} & \textbf{0.74} & \textbf{2472.63} & 14.78 & \textbf{1704.00} & 10.65 \\
$\mathrm{N}=19$ & 30\% & 54.17 & 1.02 & 2917.00 & 20.87 & 1953.58 & 14.27 \\
~ & 35\% & 56.57 & 1.18 & 3001.89 & 22.81 & 1971.68 & 15.49 \\
~ & 40\% & 53.5 & 1.08 & 3001.47 & 23.49 & 1966.00 & 14.33 \\
~ & 50\% & 53.3 & 1.15 & 2927.37 & 22.25 & 1872.42 & 13.96 \\
\midrule
~ & 10\% & 97.4 & 2.03 & 5571.16 & 54.23 & 4268.96 & 5.37 \\
~ & 15\% & 121.86 & 2.55 & 7451.60 & 84.73 & 6773.60 & 10.08 \\
~ & 20\% & 55.78 & 1.12 & 2982.47 & 21.37 & 3836.91 & 3.76 \\
\textbf{50000$\times$3000} & 25\% & 72.5 & 1.43 & 4213.29 & 36.59 & 4134.02 & 4.42 \\
$\mathrm{N}=45$ & 30\% & \textbf{39.86} & \textbf{0.84} & \textbf{2132.80} & \textbf{7.80} & \textbf{3649.84} & \textbf{3.36} \\
~ & 35\% & 106.31 & 2.10 & 5992.38 & 57.03 & 4163.18 & 5.49 \\
~ & 40\% & 103.28 & 2.10 & 5821.87 & 65.21 & 4370.24 & 9.32 \\
~ & 50\% & 151.78 & 3.04 & 8634.80 & 95.91 & 4687.00 & 7.52 \\
\bottomrule
\end{tabular}
\caption{Set A}
\end{subtable}
\hfill
\begin{subtable}[t]{0.48\textwidth}
\centering
\begin{tabular}{cccccccc}
\toprule
\textbf{(m,n)} & $\rho$ & $\bar{k}_{\mathrm{stop}}$ & $\bar{t}_{\mathrm{stop}}$ & $\bar{k}_{\mathrm{iter}}$ & $\bar{t}_{\mathrm{cmp}}$ & $\bar{c}$ & $\bar{t}_{\mathrm{comm}}$ \\
\midrule
~ & 10\% & 61 & \textbf{1.04} & 3431.31 & 23.81 & 2001.38 & 17.22 \\
~ & 15\% & 59.33 & 1.18 & 3400.88 & 24.26 & 1950.06 & 17.15 \\
~ & 20\% & 62 & 1.33 & 3439.62 & 25.51 & 1961.44 & 17.10 \\
\textbf{50000$\times$5000}  & 25\% & 59.5 & 1.14 & \textbf{3071.31} & \textbf{13.31} & 1769.31 & \textbf{10.45} \\
$\mathrm{N}=16$ & 30\% & 60.8 & 1.26 & 3293.56 & 23.30 & 1844.50 & 16.86 \\
~ & 35\% & 60.5 & 1.23 & 3302.56 & 22.87 & 1824.38 & 15.92 \\
~ & 40\% & 65.29 & 1.26 & 3321.44 & 24.36 & 1814.81 & 15.90 \\
~ & 50\% & \textbf{56.5} & 1.15 & 3238.69 & 22.91 & \textbf{1767.94} & 16.03 \\
\midrule
~ & 10\% & 73.5 & 1.41 & 4013.11 & 24.60 & 6487.09 & 6.78 \\
~ & 15\% & 56.5 & \textbf{1.04} & 3241.32 & 13.68 & 6523.62 & 6.27 \\
~ & 20\% & 121.27 & 2.41 & 7009.47 & 79.93 & 7108.40 & 10.55 \\
\textbf{60000$\times$4000} & 25\% & 65.21 & 1.37 & 3373.98 & 15.35 & 6353.72 & 5.95 \\
$\mathrm{N}=53$ & 30\% & 60.81 & 1.21 & 3086.11 & \textbf{10.28} & 6232.08 & \textbf{5.79} \\
~ & 35\% & \textbf{49.16} & 1.05 & \textbf{2627.09} & 17.96 & \textbf{4496.43} & 10.70 \\
~ & 40\% & 75.77 & 1.54 & 3882.11 & 24.03 & 6289.15 & 6.45 \\
~ & 50\% & 182.59 & 3.59 & 10136.17 & 118.97 & 7081.17 & 11.41 \\
\midrule
~ & 10\% & 60.67 & 1.20 & 3421.73 & 29.61 & 3340.69 & 18.58 \\
~ & 15\% & 62.25 & 1.18 & 3576.00 & 39.62 & 3449.12 & 24.68 \\
~ & 20\% & 59.83 & 1.17 & 3242.50 & 20.74 & 3128.42 & 12.61 \\
\textbf{80000$\times$5000} & 25\% & \textbf{50.43} & \textbf{1.01} & 2811.12 & \textbf{8.60} & 2730.38 & 7.30 \\
$\mathrm{N}=26$ & 30\% & 73.75 & 1.44 & 4155.62 & 42.31 & 3280.04 & 21.79 \\
~ & 35\% & 55.60 & 1.08 & 2817.35 & 11.69 & 2704.69 & 8.51 \\
~ & 40\% & 58.18 & 1.17 & 3202.12 & 18.75 & 3006.38 & 11.69 \\
~ & 50\% & 58.77 & 1.16 & \textbf{2800.77} & 9.46 & \textbf{2672.58} & \textbf{7.02} \\
\bottomrule
\end{tabular}
\caption{Set B}
\end{subtable}
\label{tab:failure_ratio}
\end{table}

\begin{table}[H]
\centering
\footnotesize
\caption{Performance Under Different Failure Durations Across Problem Sizes}
\begin{subtable}[t]{0.48\textwidth}
\centering
\begin{tabular}{cccccccc}
\toprule
\textbf{(m,n)} & $\xi$ & $\bar{k}_{\mathrm{stop}}$ & $\bar{t}_{\mathrm{stop}}$ & $\bar{k}_{\mathrm{iter}}$ & $\bar{t}_{\mathrm{cmp}}$ & $\bar{c}$ & $\bar{t}_{\mathrm{comm}}$ \\
\midrule
~ & 0.01\% & \textbf{17.33} & \textbf{0.18} & 2000.69 & 11.63 & 925.38 & 6.01 \\
~ & 0.02\% & 32.67 & 0.64 & 1901.15 & 9.41 & 855.92 & 5.51 \\
~ & 0.04\% & 49.67 & 1.76 & 1946.69 & 10.10 & 846.69 & 5.39 \\
\textbf{30000$\times$3000} & 0.05\% & 61.67 & 3.04 & 1873.54 & 10.61 & 775.92 & 4.76 \\
$\mathrm{N}=13$ & 0.06\% & 66.00 & 3.73 & 1890.85 & \textbf{9.24} & 782.54 & 4.51 \\
~ & 0.07\% & 67.00 & 4.57 & \textbf{1871.85} & 9.89 & 754.62 & 4.67 \\
~ & 0.09\% & 76.67 & 6.03 & 1987.15 & 10.60 & 790.69 & 5.01 \\
~ & 0.10\% & 71.33 & 7.01 & 1899.23 & 9.28 & \textbf{736.38} & \textbf{4.35} \\
\midrule
~ & 0.01\% & \textbf{26.25} & \textbf{0.24} & 3053.32 & 22.24 & 2133.00 & 15.22 \\
~ & 0.02\% & 63.25 & 1.38 & 3185.53 & 26.13 & 2163.89 & 15.88 \\
~ & 0.04\% & 81.00 & 2.77 & 2949.05 & 22.60 & 1964.89 & 14.60 \\
\textbf{40000$\times$4000}
 & 0.05\% & 101.75 & 5.09 & 2986.47 & 21.40 & 1940.26 & 14.22 \\
$\mathrm{N}=19$ & 0.06\% & 111.25 & 6.77 & 2916.68 & \textbf{20.20} & 1851.84 & 12.97 \\
~ & 0.07\% & 119.25 & 9.20 & 2958.53 & 21.97 & 1845.89 & 13.29 \\
~ & 0.09\% & 125.00 & 11.15 & 2874.58 & 20.98 & 1774.84 & 13.12 \\
~ & 0.10\% & 126.00 & 12.46 & \textbf{2845.32} & 19.25 & \textbf{1731.05} & \textbf{12.38} \\
\midrule
~ & 0.01\% & \textbf{40.67} & \textbf{0.42} & 4671.91 & 48.05 & 1963.87 & 4.73 \\
~ & 0.02\% & 44.56 & 1.06 & 2269.62 & 17.94 & 1708.53 & 5.49 \\
~ & 0.04\% & 243.00 & 8.78 & 9933.69 & 112.03 & 2670.87 & 9.98 \\
\textbf{50000$\times$3000} & 0.05\% & 185.56 & 8.91 & 5552.96 & 58.54 & 1660.07 & 5.58 \\
$\mathrm{N}=45$ & 0.06\% & 113.89 & 6.97 & 2682.18 & 23.11 & 1810.36 & 5.73 \\
~ & 0.07\% & 116.33 & 8.17 & \textbf{2202.69} & \textbf{12.83} & \textbf{1469.11} & \textbf{2.21} \\
~ & 0.09\% & 580.78 & 49.67 & 14526.20 & 167.51 & 1803.29 & 7.73 \\
~ & 0.10\% & 669.11 & 67.25 & 17094.64 & 212.40 & 1508.00 & 10.89 \\
\bottomrule
\end{tabular}
\caption{Set A}
\end{subtable}
\hfill
\begin{subtable}[t]{0.48\textwidth}
\centering
\begin{tabular}{cccccccc}
\toprule
\textbf{(m,n)} & $\xi$ & $\bar{k}_{\mathrm{stop}}$ & $\bar{t}_{\mathrm{stop}}$ & $\bar{k}_{\mathrm{iter}}$ & $\bar{t}_{\mathrm{cmp}}$ & $\bar{c}$ & $\bar{t}_{\mathrm{comm}}$ \\
\midrule
~ & 0.01\% & \textbf{31.75} & \textbf{0.31} & 3433.56 & 23.46 & 2003.62 & 16.88 \\
~ & 0.02\% & 66.00 & 1.56 & 3407.44 & 25.15 & 1925.62 & 16.30 \\
~ & 0.04\% & 97.25 & 3.31 & 3284.81 & 23.97 & 1798.81 & 15.42 \\
\textbf{50000$\times$5000} & 0.05\% & 106.00 & 5.69 & 3221.62 & 21.58 & 1715.81 & 15.06 \\
$\mathrm{N}=16$ & 0.06\% & 134.00 & 7.71 & 3340.25 & 24.29 & 1726.56 & 15.43 \\
~ & 0.07\% & 145.25 & 10.93 & 3284.75 & 23.77 & 1683.12 & 15.14 \\
~ & 0.09\% & 143.25 & 12.60 & \textbf{3154.94} & \textbf{20.52} & \textbf{1562.12} & \textbf{13.51} \\
~ & 0.10\% & 142.50 & 15.47 & 3220.06 & 21.68 & 1595.38 & 14.27 \\
\midrule
~ & 0.01\% & 88.18 & \textbf{0.91} & 9586.72 & 133.81 & 4597.49 & 21.88 \\
~ & 0.02\% & 234.27 & 5.41 & 12024.42 & 178.53 & 3762.47 & 20.03 \\
~ & 0.04\% & \textbf{77.36} & 2.74 & \textbf{2570.06} & \textbf{25.88} & \textbf{1907.89} & 9.35 \\
\textbf{60000$\times$4000} & 0.05\% & 380.73 & 18.61 & 11638.08 & 156.61 & 3397.47 & 15.83 \\
$\mathrm{N}=53$ & 0.06\% & 225.36 & 13.79 & 5277.43 & 52.96 & 2138.75 & \textbf{6.80} \\
~ & 0.07\% & 498.27 & 36.57 & 11817.42 & 141.29 & 1910.47 & 6.70 \\
~ & 0.09\% & 419.27 & 36.23 & 8874.11 & 101.80 & 2107.53 & 7.88 \\
~ & 0.10\% & 726.82 & 71.98 & 17059.15 & 227.04 & 2989.38 & 17.78 \\
\midrule
~ & 0.01\% & \textbf{34.50} & \textbf{0.35} & 3500.88 & 34.05 & 2619.46 & 24.30 \\
~ & 0.02\% & 68.50 & 1.41 & 3345.00 & 33.38 & 2486.38 & 19.53 \\
~ & 0.04\% & 84.17 & 3.31 & \textbf{2694.77} & \textbf{13.12} & \textbf{1996.00} & \textbf{9.34} \\
\textbf{80000$\times$5000} & 0.05\% & 109.17 & 5.04 & 3133.19 & 32.06 & 2252.00 & 21.94 \\
$\mathrm{N}=26$ & 0.06\% & 128.33 & 7.94 & 3218.96 & 31.42 & 2220.81 & 21.41 \\
~ & 0.07\% & 147.83 & 10.91 & 3070.00 & 30.47 & 2121.96 & 18.39 \\
~ & 0.09\% & 153.33 & 13.62 & 3082.58 & 35.00 & 2047.15 & 21.79 \\
~ & 0.10\% & 166.83 & 16.95 & 3143.23 & 34.53 & 2082.31 & 19.99 \\
\bottomrule
\end{tabular}
\caption{Set B}
\end{subtable}
\label{tab:failure_duration}
\end{table}

\textbf{Summary of Failure Robustness:} By injecting exponential-distribution-based self-halting mechanisms into a randomly selected subset of $\lceil \mathrm{N}\theta_3\rceil$ agents, we systematically evaluate how varying the \emph{failure ratio} $\theta_3 \in [0.05, 0.40]$ and \emph{failure intensity} $\xi \in [10^{-4}, 10^{-3}]$ affects global system performance. The results reveal the following consistent insights:

First, the \textbf{failure ratio is nearly linearly tolerable for small to medium systems}. For example, in configurations such as $30000 \times 3000$ ($\mathrm{N} = 13$) and $40000 \times 4000$ ($\mathrm{N} = 19$), even with $\theta_3$ increased to $0.40$, both iteration count and wall-clock time fluctuate by no more than $\pm10\%$. This confirms the inherent resilience of asynchronous architectures to localized agent failures.

However, as the problem scale grows and the agent count increases significantly (e.g., $50000 \times 3000$ with $\mathrm{N} = 45$), we observe \textbf{non-monotonic performance degradation}. When $\theta_3 \le 0.30$, the remaining active agents can efficiently compensate for halted ones. But once $\theta_3 \ge 0.35$, the communication graph becomes fragmented, diffusion paths are extended, and both iteration count and wall-clock time increase exponentially.

Second, \textbf{failure intensity exhibits a threshold effect}. When $\xi \le 2 \times 10^{-4}$ (mean halt duration $\sim$0.5s, interval $\sim$5000 iterations), all metrics degrade only mildly. In contrast, increasing $\xi$ to the range $[7 \times 10^{-4}, 10^{-3}]$ introduces frequent and prolonged halts. In larger-scale systems (e.g., $60000 \times 4000$ and $50000 \times 3000$), this causes iteration count and communication overhead to explode, with wall-clock time increasing by over $100\%$. The root cause of this inflection lies in the implicit coupling between \emph{error contraction rate} and \emph{effective communication frequency}: excessive $\xi$ prolongs local inactivity windows and accumulates residuals, undermining the contractive property of asynchronous averaging.

In summary, the algorithm demonstrates \textbf{graceful degradation under moderate failure ratios and low-to-moderate intensities}, preserving convergence without collapse. This validates its design principle of \emph{local self-healing with global feasibility}. Nevertheless, in larger-scale or high-density settings, it is crucial to impose upper bounds on both $\theta_3$ and $\xi$, or adopt finer-grained strategies for \emph{failure detection and dynamic reconnection} to sustain pre-failure convergence rates.

\paragraph{Inconsistent Systems: Sensitivity to $\lambda$ }

In this section, we investigate the influence of the regularization parameter $\lambda$ on the algorithm’s performance when solving the same inconsistent linear system. The value of $\lambda$ is varied over the following set:
\[\lambda \in \{0,\ 0.3,\ 0.7,\ 1,\ 1.3,\ 1.7,\ 2,\ 2.3,\ 2.7,\ 3\}.\]

Table~\ref{tab:不同 lambda 对算法性能的影响} reports the corresponding stopping error and average computation time, while Figure~\ref{fig:2} provides a visual comparison of convergence trajectories under different $\lambda$ values.

\begin{table}[H]
\centering
\caption{Effect of $\lambda$ on Algorithm Performance for Inconsistent Systems}
\begin{tabular}{cccc}
\toprule
\textbf{(m,n)} & $\lambda$ & $\mathrm{e}_{\mathrm{stop}}$ & $\bar{t}_{\mathrm{comm}}$ \\
\midrule
& 0   & 2.40  & 48.27 \\
& 0.3 & 1.99  & 48.95 \\
& 0.7 & 1.95  & 49.15 \\
& 1.0 & 2.93  & 53.98 \\
\textbf{60,000$\times$3,000} & 1.3 & 4.04  & 50.54 \\
$\mathrm{N}=25$ & 1.7 & 7.13  & 50.30 \\
& 2.0 & 9.50  & 49.15 \\
& 2.3 & 10.30 & 49.01 \\
& 2.7 & 14.63 & 26.72 \\
& 3.0 & 16.19 & 47.63 \\
\bottomrule
\end{tabular}
\label{tab:不同 lambda 对算法性能的影响}
\end{table}

\begin{figure}[H]
\centering
\includegraphics[width=\textwidth]{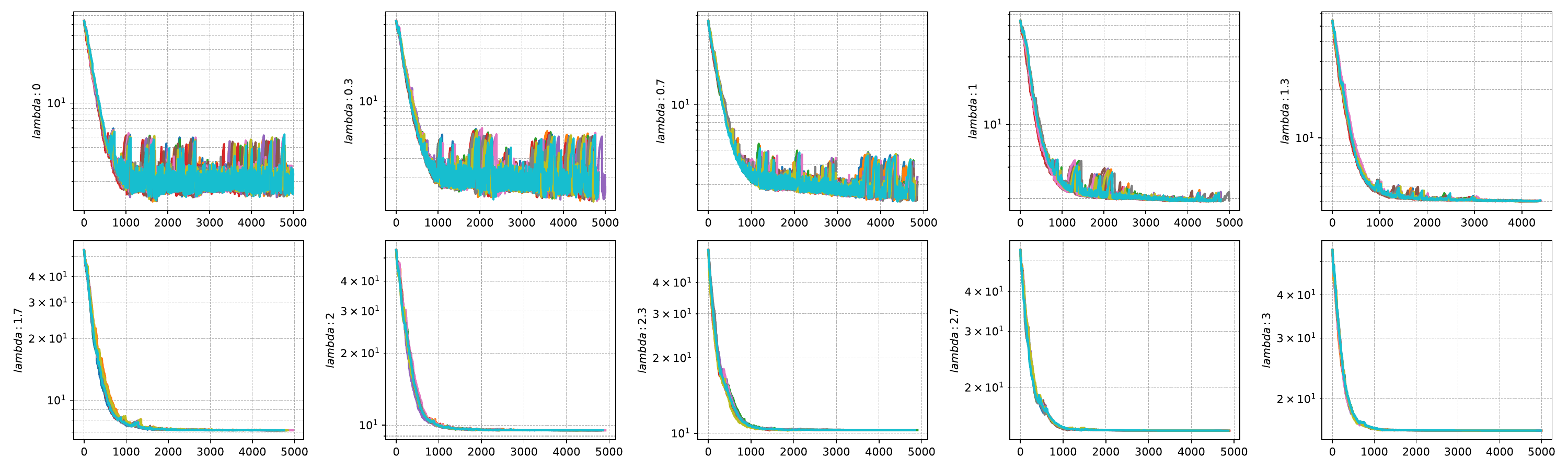}
\caption{Convergence Trajectories for Different $\lambda$ Values}
\label{fig:2}
\end{figure}

\paragraph*{Results Analysis and Insights}

Both Table~\ref{tab:不同 lambda 对算法性能的影响} and Figure~\ref{fig:2} reveal the \emph{dual role} of $\lambda$ in handling inconsistent systems: on one hand, $\lambda$ acts as a regularization term that suppresses residuals more aggressively as it increases, thereby accelerating convergence. On the other hand, a large $\lambda$ forcibly biases the search direction toward a “pseudo-consistent” subspace, leading to growing \emph{bias error}—a trend aligned with our theoretical findings in \textbf{Theorem 3}.

In detail, when $\lambda \le 0.7$, the stopping error remains below $2$, but the convergence curves exhibit mild oscillations and slow descent. Around $\lambda \approx 1.0$–$1.3$, the algorithm achieves a \emph{sweet spot} balancing speed and accuracy—converging within approximately 50 seconds with error still below $4$. As $\lambda$ increases beyond $1.7$, the stopping error escalates rapidly (reaching $14.63$ at $\lambda = 2.7$), while convergence speed plateaus, yielding diminishing returns.

\textbf{Practical Guideline:} For problems of similar scale and noise levels, we recommend a coarse search with $\lambda \in [0.7, 1.3]$, followed by fine-tuning based on specific accuracy needs. For scenarios demanding both rapid early convergence and precise final accuracy, a \emph{time-varying $\lambda$ strategy} may be beneficial—e.g., a piecewise linear decay from $\lambda = 1.3$ to $0.3$, or adaptively adjusting $\lambda$ based on residual magnitude. Such strategies preserve fast initial descent while mitigating late-stage bias, offering a promising direction for future refinement.

\section{Conclusion}

This paper presents and systematically evaluates a novel algorithmic framework for solving large-scale linear systems under resource-constrained and communication-unstable environments. The proposed method integrates event-triggered communication, asynchronous updates, and distributed randomized block Kaczmarz projection, with a focus on engineering-oriented deployment. 

First, by employing the randomized block Kaczmarz algorithm, the per-agent computational complexity per iteration is reduced from $ O(n^2)$ to $O(m_i(t_{i,k})n)$, while the newly introduced event-triggered strategy optimizes communication frequency, achieving coordinated reduction in both computational and communication overhead.

Second, a unified theoretical analysis of convergence is presented  for both consistent and inconsistent systems. For consistent systems, the algorithm achieves almost sure exponential convergence under stable communication conditions. For inconsistent systems, an augmented variable transformation converts them into consistent forms, with derived error bounds quantifying the relationship between precision and regularization parameter $\lambda$.

Furthermore, systematic experiments conducted on the ROS2 platform validate the impacts of key parameters—including node partitioning ratios, neighbor degrees, communication intervals, robustness boundaries, and regularization factor ranges—on algorithmic performance, while providing practical tuning guidelines. Experimental results demonstrate that the proposed fully asynchronous distributed Kaczmarz projection algorithm not only exhibits strong theoretical convergence guarantees but also superior scalability and stability. The implementation details and parameter-tuning strategies offer actionable engineering guidance for real-world applications in edge computing and industrial systems.

Although this work achieves both theoretical and experimental advances, several limitations remain. For instance, in inconsistent system scenarios, the current augmentation strategy only approximates a least-squares solution without exact convergence. Future work may explore residual-based variable step-size mechanisms or hybrid approaches combining augmented Lagrangian/exterior-point projections to enhance precision control. Additionally, experimental scalability was constrained by hardware resources, leaving large-scale system validation open for investigation. Possible extensions include deployment on high-bandwidth interconnected clusters or GPU/FPGA hybrid architectures, where the ROS2 communication framework can be retained while customizing underlying transport and scheduling mechanisms. Other promising directions include developing residual/bandwidth-driven adaptive event-triggering, ultra-large-scale implementations integrating randomized block methods with GPU parallelism, and designing safety-robust mechanisms against node failures or unreliable links.

\section*{Appendix: Proof of Theorem}

The proof of \textbf{Theorem 2} is based on constructing a finite-length sequence of agent states ${x_i(\tau_s), x_i(\tau_{s-1}), \dots, x_i(\tau_{s-d})}$, thereby extending the analysis from individual time points to multiple consecutive time steps. This temporal expansion enables the capture of asynchronous characteristics that influence global system dynamics.

Leveraging the asynchronous modeling framework proposed in~\cite{Liu2018Asynchronous}, the evolution of error sequences across agents can be formalized into a global state transition equation. The transition matrix is block-structured, with each block corresponding to an orthogonal projection operator defined over a subspace.

To accommodate both full-rank and rank-deficient cases, we decompose the global state space structurally. One subset of state variables evolves independently of the projection operators and can be interpreted as a multi-agent consensus process under asynchronous updates. Its stability and convergence have been thoroughly analyzed in~\cite{doi:10.1137/060657005}. The remaining variables are directly influenced by the projection matrices; their convergence behavior forms the crux of the proof and reduces to an asynchronous random projection consensus problem.

The convergence analysis for such systems follows the technical blueprint laid out in~\cite{Liu2018Asynchronous} and~\cite{Yi2020Distributed}. The former addresses projection consensus under asynchronous mechanisms, while the latter analyzes random projection systems in full-rank scenarios. Both approaches rely on establishing a strict error contraction condition—namely, demonstrating that, in expectation, the product of projection matrices shrinks system error. However, existing results do not explicitly address convergence under rank deficiency, necessitating extensions to the theoretical framework.

An additional technical challenge arises from the need to integrate two stochastic components: random block selection and time-varying communication topologies. These must be combined to construct a state transition matrix that guarantees error decay. The strong connectivity of the communication graph is characterized in~\cite{Yi2020Distributed} using the *-mixing property of stochastic graph processes. Under asynchronous conditions, traditional static graphs fail to capture the true dependency structure of the system.

To overcome this, we adopt the delay-graph modeling technique introduced in~\cite{Liu2018Asynchronous}, where the original system is expanded into an augmented graph with $d\mathrm{N}$ nodes. Each agent’s behavior at different delay stages is represented by a separate node. Using advanced graph-theoretic tools—such as strong rootedness and quotient graphs—we can rigorously characterize the strong connectivity of this augmented structure and establish the required joint connectivity events.

\textbf{Summary of Proof Strategy and Contributions:} In summary, the core structure of the proof hinges on four key components:

(1) \textbf{Unified State-Space Modeling:} We first construct a unified state-space representation of the system and perform a structural decomposition, isolating the projection-driven subdynamics as the primary object of analysis.

(2) \textbf{Generalized Contraction Conditions:} Then extend the strict error contraction criterion to accommodate rank-deficient scenarios, thereby establishing convergence under more general conditions.

(3) \textbf{Probabilistic Connectivity Guarantees:} Next, we use the theorem’s assumptions to construct probabilistic events that jointly ensure error reduction, including completeness of random block selection and strong connectivity of the augmented delay graph.

(4) \textbf{Integration with Non-Projection Dynamics:} Finally, we incorporate the non-projection component of the system evolution to complete a full convergence proof in the asynchronous setting.

The main technical contribution of this proof lies in its formulation of a convergence condition applicable to rank-deficient systems, thereby advancing the convergence theory of random projection-based systems under asynchronous updates. Furthermore, we introduce a communication design principle grounded in probabilistic language: the key convergence condition is translated into a practical engineering constraint that long-term communication failure must occur with zero probability.

Detailed technical development is provided in the subsections \textit{State-Space Formulation}, \textit{Strict Contraction Analysis}, and \textit{Graph-Theoretic Foundations}, which together form a complete and rigorous convergence proof.

\subsection{State-Space Formulation}

This section derives the state transition equation for the delayed sequence and presents its decomposition. Since the delay $\delta_{j,i}(t_{i,k})$ is bounded, we denote the global delay upper bound by $\delta$. Let $T = \min\limits_i {T_i}$ be the minimum local iteration interval. Then the maximum delay length is given by $d = \left\lceil \frac{\delta}{T} \right\rceil$.

Let $x^* = A^\dagger b$ denote the minimum-norm solution to the consistent system. The iteration update satisfies:
\[\begin{aligned}
x_i(t_{i,k+1}) & - x^* = \frac{1}{d_i(t_{i,k})} w_i(t_{i,k}) - x^*  - A_i^\dagger(t_{i,k}) A_i(t_{i,k}) \left( \frac{1}{d_i(t_{i,k})} w_i(t_{i,k}) - x^* \right) \\
&\quad + A_i^\dagger(t_{i,k}) b_i(t_{i,k}) - A_i^\dagger(t_{i,k}) A_i(t_{i,k}) x^* = P_i(t_{i,k}) \left( \frac{1}{d_i(t_{i,k})} w_i(t_{i,k}) - x^* \right),
\end{aligned}\]
where $P_i(t_{i,k}) = I - A_i^\dagger(t_{i,k}) A_i(t_{i,k})$ denotes the projection operator.

Define the local error $e_i(\tau_s) = x_i(\tau_s) - x^*$, the global error vector $e(\tau_s) = \big[e_1(\tau_s)^T, \dots, e_N(\tau_s)^T\big]^T \in \mathbb{R}^{Nn}$, and the delay-augmented error vector $\mathbf{e}(\tau_s) = \big[e(\tau_s)^T, \dots, e(\tau_{s-d})^T\big]^T \in \mathbb{R}^{dNn}$.

Define the time-dependent projection
matrix for agent $i$ as:
\[P_i(\tau_k) =
\begin{cases}
P_i(t_{i,k}), & \text{if } t_{i,k} = \tau_k, \\
I_n, & \text{otherwise}.
\end{cases}\]

The communication weight matrix $W_{i,j}(\tau_k)$ is defined as:
\[W_{i,j}(\tau_k) =
\begin{cases}
\frac{1}{m_i(\tau_k)}, & \text{if } \exists p \in \mathcal{N}_i(\tau_k),\ l \in \{1,\dots,d\}, \text{s.t. } j = p + lN, \ i \text{ receives from } p \text{ at } \tau_{k-l}, \\
1, & \text{if } i = p + lN,\ j = p + (l{-}1)N, p \in \{1,\dots,N\},\ l \in \{1,\dots,d\}, \\
0, & \text{otherwise}.
\end{cases}\]

This matrix $W(\tau_k)$ coincides with the delay-graph formulation introduced in~\cite{Liu2018Asynchronous}, where each node-delay pair is represented as a distinct node in an augmented graph with $dN$ total nodes.

Let $P(\tau_k) = \text{diag}(P_1(\tau_k), \dots, P_N(\tau_k))$ and define the block-diagonal projection matrix:
\[\mathbf{P}(\tau_k) = \text{diag}(P(\tau_k), P(\tau_{k-1}), \dots, P(\tau_{k-d})),\]
noting that $\mathbf{P}^2(\tau_k) = \mathbf{P}(\tau_k)$.

The update rule can now be rewritten in compact form as:
\begin{equation}
\mathbf{e}(\tau_{k+1}) = \mathbf{P}(\tau_k) \left(W(\tau_k) \otimes I_n\right) \mathbf{P}(\tau_{k-1}) \mathbf{e}(\tau_k),
\label{eq:未分解的系统}
\end{equation}
which constitutes a linear time-varying autonomous discrete-time system.

\subsection{Subspace Decomposition and Convergence Matrix Construction}

Since the matrix $A$ is general (i.e., not necessarily full rank), the sequential product of projection matrices constructed from its row blocks cannot ensure strict contraction for error vectors across the full space. However, such contraction can still be achieved within specific subspaces. Therefore, we decompose the system~\eqref{eq:未分解的系统} accordingly.

Let $\mathbb{R}^n = \text{Row}(A) \oplus \text{Row}_\perp(A)$ denote the orthogonal decomposition of the ambient space. Consequently, the full delayed error space $\mathbb{R}^{dNn}$ can also be decomposed as $\left[\text{Row}(A)\right]^{dN} \oplus \left[\text{Row}_\perp(A)\right]^{dN}$. This yields a decomposition of the error vector:
\[\mathbf{e}(\tau) = \mathbf{e}_1(\tau) + \mathbf{e}_2(\tau),\]
where $\mathbf{e}_1(\tau) \in \left[\text{Row}(A)\right]^{dN}$ and $\mathbf{e}_2(\tau) \in \left[\text{Row}_\perp(A)\right]^{dN}$.

Under this decomposition, the system~\eqref{eq:未分解的系统} separates into two independent subsystems:
\[\begin{aligned}
\mathbf{e}_1(\tau_{k+1}) &= \mathbf{P}(\tau_k)\left(W(\tau_k) \otimes I_n\right) \mathbf{P}(\tau_{k-1}) \mathbf{e}_1(\tau_k), \\
\mathbf{e}_2(\tau_{k+1}) &= \mathbf{P}(\tau_k)\left(W(\tau_k) \otimes I_n\right) \mathbf{P}(\tau_{k-1}) \mathbf{e}_2(\tau_k).
\end{aligned}\]

Let $\mathcal{I} \subseteq \mathbf{m}$ denote any subset of the row index set of $A$. Define $A_{\mathcal{I}}$ as the submatrix formed by the corresponding rows of $A$. Since $\text{Row}_\perp(A)$ is the null space of $A$, for any $\mathbf{y} \in \text{Row}_\perp(A)$ and any $A_{\mathcal{I}}$, the projection identity holds:
\[\left(I_n - A_{\mathcal{I}}^\dagger A_{\mathcal{I}}\right) \mathbf{y} = \mathbf{y}.\]
Hence, the second subsystem simplifies to:
\begin{equation}
\mathbf{e}_2(\tau_{k+1}) = \left(W(\tau_k) \otimes I_n\right) \mathbf{e}_2(\tau_k),
\label{eq:普通共识系统}
\end{equation}
which describes a classical consensus system over a time-varying delay graph.

Next, let $P_A(\tau_k)$ denote the restriction of $P(\tau_k)$ onto $\text{Row}(A)$, and define the block-diagonal projection:
\[\mathbf{P}_A(\tau_k) = \text{diag}\left(P_A(\tau_k), \dots, P_A(\tau_{k-d})\right),\]
which acts on $\left[\text{Row}(A)\right]^{dN}$. Then the first subsystem becomes:
\begin{equation}
\mathbf{e}_1(\tau_{k+1}) = \mathbf{P}_A(\tau_k)\left(W(\tau_k) \otimes I_n\right)\mathbf{P}_A(\tau_{k-1}) \mathbf{e}_1(\tau_k),
\label{eq:投影共识系统}
\end{equation}
representing a \emph{projected consensus system} on the active subspace.

To analyze convergence, we focus on the spectral properties of the composite transition operator:
\begin{equation}
 M_A = \mathbf{P}_A(\tau_k)\left(W(\tau_k) \otimes I_n\right)\mathbf{P}_A(\tau_{k-1}) \left(W(\tau_{k-1}) \otimes I_n\right) \cdots \mathbf{P}_A(\tau_1)\left(W(\tau_1) \otimes I_n\right)\mathbf{P}_A(\tau_0),
\label{eq:状态转移矩阵}
\end{equation}
which is a block matrix composed of interleaved projection operators restricted to $\text{Row}(A)$. The spectral radius of $M_A$ will serve as the basis for establishing exponential convergence on the subspace $\text{Row}(A)$.

\subsection{Strict Contraction Conditions}

To prove that the spectral radius of $M_A$ is strictly less than 1, a common approach is to identify a product of projection matrices—hereafter referred to as a \textit{projection matrix polynomial}—whose spectral norm is strictly contractive under certain conditions. We then construct such matrix products within $M_A$ by leveraging the connectivity structure of the communication graph.

In our setting, due to the generality of $A$, we restrict analysis to the subspace $\left[\text{Row}(A)\right]^{dN}$, where strict contraction can still be achieved. This requires defining an appropriate norm structure over this subspace.

A hybrid matrix norm $\|\cdot\|$ has been utilized in prior work~\cite{Liu2017Exponential, Liu2018Asynchronous, Yi2020Distributed} to establish convergence of asynchronous systems. For a block matrix $A = [A_{ij}]_{dm \times dm}$ with $A_{ij} \in \mathbb{R}^{n \times n}$,the hybrid norm of $A$ is defined as:
\[\|A\| = \left\| [\|A_{ij}\|_2] \right\|_\infty.\]
We now define a corresponding norm over $\left[\text{Row}(A)\right]^{dN}$, denoted $\|\cdot\|_A$, as follows.

For a linear operator $P$ acting on $\text{Row}(A)$, define the induced norm:
\[\|P\|_2^A = \max \left\{ \|Px\|_2 : x \in \text{Row}(A),\ \|x\|_2 = 1 \right\}.\]
Then for any block matrix $P = [P_{ij}]_{dN \times dN}$ with $P_{ij} \in \mathbb{R}^{n \times n}$, define the hybrid operator norm on $\left[\text{Row}(A)\right]^{dN}$ as:
\[\|P\|_A = \left\| [\|P_{ij}\|_2^A] \right\|_\infty.\]
It can be shown that $\|\cdot\|_A$ defines a valid operator norm on $\left[\text{Row}(A)\right]^{dN}$, and for any $x = [x_1^T, \dots, x_{dN}^T]^T$ with $x_i \in \text{Row}(A)$, we have:
$\|x\|_A = \max\limits_i \|x_i\|_2.$

Let $P_{\mathcal{I}} = I_n - A_{\mathcal{I}}^\dagger A_{\mathcal{I}}$ denote the orthogonal projection onto the null space of $A_{\mathcal{I}}$, where $\mathcal{I} \subseteq \mathbf{m}$ is a row index subset. Let $P_{\mathcal{I}}^A$ denote the restriction of $P_{\mathcal{I}}$ to $\text{Row}(A)$.

A \textbf{projection matrix polynomial on $\text{Row}(A)$} is defined as:
\[\mu = \sum_{q=1}^c \lambda_q P_{\mathcal{I}^{q}_{n_q}}^A \cdots P_{\mathcal{I}^q_1}^A,\]
where $c$ is the number of terms, each weighted by a positive coefficient $\lambda_q > 0$, and each term is a product of $n_q$ projection operators. The sets $\mathcal{I}_i^q,i=1,\cdots, n_q$ specify the index sets of row blocks involved in each projection.

Define the weight of the polynomial as:
$\lceil \mu \rceil = \sum\limits_{q=1}^c \lambda_q.$
It can be verified that this operator satisfies:
\[\lceil \mu_1 \mu_2 \rceil = \lceil \mu_1 \rceil \lceil \mu_2 \rceil, \quad \lceil \mu_1 + \mu_2 \rceil = \lceil \mu_1 \rceil + \lceil \mu_2 \rceil, \quad \|\mu\|_2^A \le \lceil \mu \rceil.\]

Let $\mathcal{L}(A)$ denote the family of all index sets corresponding to maximal linearly independent row subsets of $A$. That is, for any $S \in \mathcal{L}(A)$, the rows indexed by $S$ form a basis for the row space of $A$. To ensure that $\|\mu\|_2^A < \lceil \mu \rceil$, we next extend \textbf{Lemma 2} in~\cite{Liu2018Asynchronous}, generalizing it to arbitrary-rank $A$.

\begin{lemma} Let \( \mathbf{I} \) be a set composed of some subsets of index set \( \mathbf{m} \). If there exists \( S \in \mathcal{L}(A) \) such that \( S \subseteq \bigcup\limits_{\mathcal{I} \in \mathbf{I}} \mathcal{I} \), then:
\[
\left\|\prod\limits_{\mathcal{I} \in \mathbf{I}} P_{\mathcal{I}}^A \right\|_2^A < 1.
\]
\end{lemma}

\begin{proof}
    For any \( S \in \mathcal{L}(A) \), it holds that \( \text{Row}(A) = \text{span} \left\{ \alpha_h \ \middle| \ h \in S \right\} \).  
    By the non-expansive property of projection operators, for any \( x \in \text{Row}(A) \) with \( \|x\|_2=1 \), we have:
    \[
    \left\|\left( \prod\limits_{\mathcal{I} \in \mathbf{I}} P_{\mathcal{I}}^A \right) x \right\|_2 \leq 1.
    \]
    
    Define \( \prod\limits_{\mathcal{I} \in \mathbf{I}} P_{\mathcal{I}}^A=P_{\mathcal{I}_l}^A\cdots P_{\mathcal{I}_1}^A \).  
    If the equal sign in the above inequality holds, then due to the non-increasing nature of projection operators, for any \( 2 \leq s \leq l \), we must have:
    \[
    \left\|\left(\overset{s}{\prod\limits_{t=1}}P_{\mathcal{I}_t}^A\right)x\right\|_2=\left\|\left(\overset{s-1}{\prod\limits_{t=1}}P_{\mathcal{I}_t}^A\right)x\right\|_2=\|x\|_2=1.
    \]
    We now use induction to show that in this case,with \( x \in \text{Row}(A) \) and \( \|x\|_2=1 \), it must have $x\perp\alpha_h, h\in\overset{l}{\bigcup\limits_{t=1}}\mathcal{I}_t $.

    First, for $s=1$, we can immediately obtain that  $x\perp\alpha_h, \\ h\in\mathcal{I}_1$.  
    Now assume that for some \( s-1 \), we have $x\perp\alpha_h,h\in\overset{s-1}{\bigcup\limits_{t=1}}\mathcal{I}_t.$  
    Then it follows from 
    $\left\|\left(\overset{s}{\prod\limits_{t=1}}P_{\mathcal{I}_t}^A\right)x\right\|_2=\left\|\left(\overset{s-1}{\prod\limits_{t=1}}P_{\mathcal{I}_t}^A\right)x\right\|_2$
     that $\left(\left(\overset{s-1}{\prod\limits_{t=1}}P_{\mathcal{I}_t}^A\right)x\right)\perp\alpha_h, h\in\mathcal{I}_s.$
    Note that  \( \left(\overset{s-1}{\prod\limits_{t=1}}P_{\mathcal{I}_t}^A\right)x = x \), it then follows 
     $x\perp\alpha_h,h\in\mathcal{I}_s.$
    Hence, $x\perp\alpha_h,h\in\overset{s}{\bigcup\limits_{t=1}}\mathcal{I}_t.$
    By induction, the claim follows.  
    
    Since \( S \subseteq \bigcup\limits_{\mathcal{I} \in \mathbf{I}}\mathcal{I} \), it follows that  
    \[
    x \in \text{Row}(A)\cap \text{Row}_\perp(A),
    \]
    which implies \( x=0 \), this is  a contradiction.  
    Thus, for any \( x \in \text{Row}(A) \) with \( \|x\|_2=1 \), we obtain:
    \[
    \left\|\left( \prod\limits_{\mathcal{I} \in \mathbf{I}} P_{\mathcal{I}}^A \right) x \right\|_2 < 1. \text{ Hence, }
    \left\|\prod\limits_{\mathcal{I} \in \mathbf{I}} P_{\mathcal{I}}^A \right\|_2^{A} < 1.
    \]
\end{proof}

\subsection{Sufficient and Necessary Conditions for Strict Contraction}

\textbf{Lemma 1} establishes that a projection matrix polynomial $\mu$ satisfies the strict inequality $|\mu|_2^A < \lceil \mu \rceil$ if and only if there exists at least one term in $\mu$ whose constituent projection matrices collectively cover a maximal linearly independent subset of the rows of $A$. We refer to such polynomials as \textbf{complete}.

Consider a block matrix $U = [\mu_{ij}]_{dm \times dm}$ composed of projection matrix polynomials. Define its coefficient-weighted counterpart as $\lceil U \rceil = [\lceil \mu_{ij} \rceil]_{dm \times dm}$. Due to the properties $\lceil \mu_1 \mu_2 \rceil = \lceil \mu_1 \rceil \lceil \mu_2 \rceil$ and $\lceil \mu_1 + \mu_2 \rceil = \lceil \mu_1 \rceil + \lceil \mu_2 \rceil$, it follows that:
\[\lceil U_1 U_2 \rceil = \lceil U_1 \rceil \lceil U_2 \rceil.\]
Note that the operator $M_A$ defined in~\eqref{eq:状态转移矩阵} is precisely a matrix of this form.

We now provide a criterion for strict contraction under the hybrid norm:

\begin{proposition}
Let $U = [\mu_{ij}]_{dm \times dm}$ be a matrix of projection matrix polynomials on $\text{Row}(A)$. Then:
\[
\begin{aligned}
    \|U\|_A < & \|\lceil U \rceil\|_\infty \iff \text{Each row of } U \\
& \text{ contains at least one complete polynomial}.
\end{aligned}
\]
\end{proposition}

\begin{proof}
According to \textbf{Lemma 1}, each $\mu_{ij}$ satisfies $\|\mu_{ij}\|_2^A < \lceil \mu_{ij} \rceil$ if and only if it is complete. Therefore, the condition $\|U\|_A < \|\lceil U \rceil\|_\infty$ is equivalent to requiring that every row $i$ of $U$ contains at least one entry $\mu_{ij}$ that is complete.
\end{proof}

Now, observe that $\lceil \mathbf{P}_A(\tau) \rceil = I_{dm \times dm}$, it follows that $\lceil M_A \rceil = \lceil W(\tau_k) \rceil \cdots \lceil W(\tau_1) \rceil$, and because each $W(\tau)$ is a row-stochastic matrix, we have:
\[\|\lceil M_A \rceil\|_\infty = 1.\]
Therefore, by \textbf{Proposition 1}, if $M_A$ satisfies the completeness condition on every row, it follows that:
\[\|M_A\|_A < 1.\]

This guarantees exponential contraction of the projected consensus system~\eqref{eq:投影共识系统} in $\left[\text{Row}(A)\right]^{dN}$ whenever such a condition is repeatedly met during the iterative process.

\subsection{Graph-Theoretic Foundations}

In the previous section, we established sufficient conditions under which $\|M_A\| < 1$. Notably, the composition of each row of the block projection polynomial matrix $M_A$ is closely related to the matrix $W(\tau_k)$, and consequently, to the structure of the communication graph. This section delves into the graph-theoretic properties of the communication network to construct convergence guarantees. The theoretical framework closely follows~\cite{Liu2018Asynchronous}, which we revisit here for completeness.

\paragraph{Delayed Graph}

A \textit{delayed graph} is a directed graph defined at each time step $\tau_k$ as $\mathcal{G}_d(\tau_k) = (\mathcal{N}_d, \mathcal{V}_d(\tau_k))$, where $\mathcal{N}_d$ and $\mathcal{V}_d(\tau_k)$ denote the node and edge sets, respectively. For a system with delay $d$, the node set for agent $i$ is defined as $\mathcal{N}_i = \{i_0, i_1, \ldots, i_d\}$, where each $i_s$ represents the $s$-th delayed state of agent $i$, with $i_0 = i$. The overall node set is $\mathcal{N}_d = \bigcup\limits_{i \in \mathcal{N}} \mathcal{N}_i$.

The edge set $\mathcal{V}_d(\tau_k)$ is defined as follows. Given $\mathcal{A}(\mathcal{G}(\tau_k)) = W(\tau_k)$, it include the edges $(i_s, i_{s+1})$ and $(i, i)$ for all $i$. Additionally, if agent $i$ utilizes the $s$-th delayed state of agent $j$ at time $\tau_k$, then $(j_s, i)$ is also included in $\mathcal{V}_d(\tau_k)$.

\paragraph{Quotient Graph and Agent Subgraph}

Define the \textit{quotient graph} $\mathcal{Q}(\tau)$ of $\mathcal{G}_d(\tau)$ to be a graph over the original node set $\mathcal{N}$. For agents $i$ and $j$, an edge $(i, j)$ exists in $\mathcal{Q}(\tau)$ if there exists $i_s \in \mathcal{N}_i$ and $j_t \in \mathcal{N}_j$ such that $i_s$ is a neighbor of $j_t$ in $\mathcal{G}_d(\tau)$. The \textit{agent subgraph} $\mathcal{D}(\tau)$ is defined as the induced subgraph of $\mathcal{G}_d(\tau)$ on the node subset $\mathcal{N}$.

\paragraph{Root and Hierarchical Graphs}

A node $v$ in a directed graph $\mathcal{G}$ is called a \textit{root} if there exists a path from $v$ to every other node in the graph. A directed graph $\mathcal{G}$ is a \textit{hierarchical graph} if its nodes can be ordered as ${v_1, v_2, \ldots, v_n}$ such that:
(1) $v_1$ is a root, and
(2) for any $i < j$, if there exists a path from $v_i$ to $v_j$, then $j > i$.

Such an ordering is called a \textit{hierarchical structure} of $\mathcal{G}$. A hierarchical graph is said to have a \textit{strong root structure} if the root has direct edges to all other nodes. An extended delayed graph is said to have a \textit{strong root hierarchy} if every induced subgraph on the subset $\mathcal{N}_i$ has a strong root structure.

\paragraph{Paths in a Graph Sequence}

Given a sequence of directed graphs $\mathcal{G}_1, \mathcal{G}_2, \ldots, \mathcal{G}_k$, a node sequence $j = v_{i_0}, v_{i_1}, \ldots, v_{i_k} = i$ is said to form a \textit{path} if $(v_{i_{s-1}}, v_{i_s})$ is an edge in $\mathcal{G}_s$ for each $s$. If such a path exists, then the composed graph $\mathcal{G}_k \circ \cdots \circ \mathcal{G}_2 \circ \mathcal{G}_1$ contains the edge $(i, j)$.

\paragraph{Joint Strong Connectivity}

A sequence of directed graphs $\mathcal{G}_1, \mathcal{G}_2, \ldots, \mathcal{G}_k$ is said to be \textit{$l$-strongly connected} if every contiguous subsequence of length $l$ forms a composed graph $\mathcal{G}_{s+l} \circ \cdots \circ \mathcal{G}_{s+1}$ that is strongly connected.

\paragraph{Key Lemmas}

The following lemmas are stated without proof, as presented in~\cite{Cao2008Reaching, Liu2018Asynchronous, Yi2020Distributed}:

\begin{lemma}
    Any combination of at least $d$ delayed graphs has a strong root hierarchy.
\end{lemma}

\begin{lemma}
    Let $\mathcal{G}_d^p$ and $\mathcal{G}_d^q$ be delayed graphs. If $\mathcal{G}_d^p$ has a strongly connected quotient graph and $\mathcal{G}_d^q$ has a strong root hierarchy, then $\mathcal{G}_d^p \circ \mathcal{G}_d^q$ has a strongly connected agent subgraph.
\end{lemma}

\begin{lemma}
    The combination of more than $\mathrm{N}-1$ strongly connected directed graphs with self-loops at every node forms a fully connected graph.
\end{lemma}

\begin{lemma}
    If $n \geq (\mathrm{N}-1)(d+1)$, then any sequence of $n$ delayed graphs with strongly connected quotient graphs yields a fully connected agent subgraph.
\end{lemma}

\begin{lemma}
    Any combination of more than $(d+1)(\mathrm{N}-1) + d$ delayed graphs with strongly connected quotient graphs ensures a strong root structure.
\end{lemma}

\begin{lemma}
    If a path $v^1_{d_1}, \ldots, v^k_{d_k}$ exists in the extended delayed graph sequence $\mathcal{G}_d(\tau_1), \ldots, \mathcal{G}_d(\tau_k)$, then the $(d_k\mathrm{N} + v^k, d_1\mathrm{N} + v^1)$ entry of matrix $M_A$ includes the product $\prod\limits_{t=1}^k P^A_{v^t}(\tau_{t-d_t})$.
\end{lemma}

\begin{remark}
\textbf{Lemma 6} is adapted from \textbf{Proposition 4} of~\cite{Cao2008Reaching}. While the original proposition requires $(d+1)(\mathrm{N}-1)^2 + d$ delayed graphs for root structure under strong root assumption, this bound reduces to $(d+1)(\mathrm{N}-1) + d$ under the weaker assumption of strongly connected quotient graphs.
\end{remark}

\subsection{Theorem Proof}

To construct the condition under which $M_A$ is strictly contractive, we first define a toolset $\mathcal{P}^i_n$. Each element of this set is a sequence of projection matrices, restricted to $\text{Row}(A)$, generated from the row blocks of $A_i$, with a sequence length of at least $n$. Any sequence in $\mathcal{P}^i_n$ satisfies the property that for every consecutive subsequence of length greater than $n$, the union of row indices used by the projection matrices in this subsequence covers all row vectors of $A_i$.

Based on the above tools, we present two propositions to prove the convergence of systems \eqref{eq:投影共识系统} and \eqref{eq:普通共识系统}, respectively.

\subsubsection{Convergence of the Projection Consensus System}

\begin{proposition}
    If the matrices $W(\tau_k), \cdots, W(\tau_1)$ involved in $M_A$, the associated delayed graph sequence $\mathcal{G}_d(\tau_k),\cdots,\mathcal{G}_d(\tau_1)$ generated by them and the projection matrix sequences $P_i^A(\tau_k), \cdots, P_i^A(\tau_0)$ for each agent $i$ satisfy the following conditions:

(1) There exists a positive integer $l$ such that the graph sequence $\mathcal{G}_d(\tau_k),\cdots,\mathcal{G}_d(\tau_1)$ is $l$-strongly connected.

(2) For every $i$, there exists a positive integer $\kappa$ such that the projection sequence $P_i^A(\tau_k), \cdots, P_i^A(\tau_0)$ belongs to $\mathcal{P}^i_\kappa$.

(3) Let $\varrho$ be the least common multiple of $\kappa$ and $(d+1)l$, and suppose $k \geq \big(N(N-1)+2\big)\varrho$.

Then $M_A$ is a strictly contractive matrix.
\end{proposition}

\begin{proof}
The proof strategy primarily hinges on \textbf{Lemma 7}, which connects the path structures within the delayed graph to the polynomial components of each block projection in $M_A$. By leveraging the $l$-strong connectivity of the delayed graph sequence, together with \textbf{Lemmas 3, 4,} and \textbf{5}, and the fact that $P_i^A(\tau_k), \cdots, P_i^A(\tau_0) \in \mathcal{P}^i_\kappa$, we identify paths within $\mathcal{G}_d(\tau_k), \cdots, \mathcal{G}_d(\tau_1)$ that ensure each block row of $M_A$ satisfies \textbf{Proposition 1}, thus confirming the strict contraction property of $M_A$.

Let $\omega = N\varrho$, and partition the graph sequence $\mathcal{G}_d(\tau_1), \cdots, \mathcal{G}_d(\tau_k)$ into subgroups $\mathbf{G}_i = \left\{ \mathcal{G}_d(\tau_{(i-1)\omega+1}), \cdots, \mathcal{G}_d(\tau_{i\omega}) \right\}$, for $i \in \{1, 2, \cdots, N-1\}$, and $\mathbf{G}_N = \left\{ \mathcal{G}_d(\tau_{(N-1)\omega+1}), \cdots, \mathcal{G}_d(\tau_s) \right\}$. Each $\mathbf{G}_i$ is subdivided into $\mathbf{G}_i^1$ and $\mathbf{G}_i^2$, likewise for $\mathbf{G}_N$.

In $\mathbf{G}_i^1$, due to the presence of self-loops on each node in the delayed graph’s subset $\mathcal{N}$, every node $g \in \mathcal{N}$ admits an identity walk $\mathcal{O}_{i1}^g$ of length $\varrho+1$ entirely composed of node $g$.

Next, consider $\mathbf{G}_i^2$. Since the original graph sequence is $l-$strongly connected, \textbf{Lemmas 3} and \textbf{4} imply that each composite graph $\mathcal{G}_d(\tau_{(i-1)\omega+(j+1)\varrho}) \circ \cdots \circ \mathcal{G}_d(\tau_{(i-1)\omega+j\varrho+1})$ is strongly connected. By \textbf{Lemma 5}, the larger composition $\mathcal{G}_d(\tau_{i\omega}) \circ \cdots \circ \mathcal{G}_d(\tau_{(i-1)\omega+\varrho+1})$ forms a complete subgraph. Hence, for any pair $g, h \in \mathcal{N}$, there exists a path $\mathcal{O}_{i2}^{g,h}$ on $\mathbf{G}_i^2$ linking $g$ to $h$.

A similar reasoning applies to $\mathbf{G}_N$. Each node $h \in \mathcal{N}$ supports an identity walk $\mathcal{O}_{N1}^h$ on $\mathbf{G}_N^1$. Given that the length of $\mathbf{G}_N^2$ exceeds $d$, \textbf{Lemma 2} ensures the existence of a strongly rooted hierarchical structure. For any $h \in \mathcal{N}$ and its descendant $h_s \in \mathcal{N}_h$, a corresponding path $\mathcal{O}_{N2}^{h,h_s}$ exists on $\mathbf{G}_N^2$.

Now, for any node $g \in \mathcal{N}$ and index $s \in \{1,2,\cdots,d\}$, construct a node visit sequence $h = i_1, i_2, \cdots, i_N = g$, where $h \ne g$, and all other elements form a permutation of $\mathcal{N}$. For each $k \in \{1,\cdots,N-1\}$, identify the corresponding identity walk $\mathcal{O}_{i1}^{i_k}$ and inter-node path $\mathcal{O}_{i2}^{i_k,i_{k+1}}$. Also include the final identity walk $\mathcal{O}_{N1}^{i_N}$ and terminal path $\mathcal{O}_{N2}^{i_N,{i_N}_s}$. Concatenating these paths yields the overall composite path.

By \textbf{Lemma 7}, the projection polynomial in the block row $sN+g$, block column $h$ of $M_A$ contains the product $\overset{\varrho}{\prod\limits_{s=1}} P^A_{i_1}(\tau_s)$ $\cdots$ $\overset{\varrho}{\prod\limits_{s=(j-1)\omega+1}} P^A_{i_j}(\tau_s)$ $\cdots$ $\overset{\varrho}{\prod\limits_{s=(N-2)\omega+1}} P^A_{i_{N-1}}(\tau_s) $ $\overset{\varrho}{\prod\limits_{s=(N-1)\omega+1}} P^A_{i_N}(\tau_s)$ $\cdots$. Since all $P_i^A(\tau_j) \in \mathcal{P}_\kappa^i$, and $g, s$ are arbitrary, \textbf{Proposition 1} implies that $\|M_A\|_A < 1$, ensuring that the spectral radius of $M_A$ is strictly less than one. Given the finiteness of all involved graph and matrix sets, a uniform constant $\gamma$ exists:
\[
\gamma =  \sup\limits_{\mathbf{G}_1^1} \cdots \sup\limits_{\mathbf{G}_{N-1}^1}\sup\limits_{\mathbf{G}_1^2} \cdots \sup\limits_{\mathbf{G}_{N-1}^2}\sup\limits_{\mathbf{G}_1^0}\sup\limits_{\mathbf{G}_2^0} \Big\|\mathbf{P}_A(\tau_s) \left(W(\tau_s) \otimes I_n\right) \mathbf{P}_A(\tau_{s-1})(W(\tau_{s-1})  \otimes I_n) \cdots \mathbf{P}_A(\tau_{1}) \left(W(\tau_{1}) \otimes I_n\right) \mathbf{P}_A(\tau_0) \Big\|_A < 1.
\]
\end{proof}

\subsubsection{Convergence of the Standard Consensus System}
We next analyze the convergence of system \eqref{eq:普通共识系统}, which models a delayed consensus problem over $\mathrm{N}$ agents in $\mathbb{R}^n$. Such problems have been extensively studied in the literature, including~\cite{doi:10.1137/060657005} and~\cite{Cao2008Reaching}. In the proposed formulation, the system \eqref{eq:普通共识系统} can be viewed as a time-inhomogeneous Markov chain, with the state transition matrix given by $W(\tau_k) \otimes I_n$. Therefore, the convergence analysis is equivalent to studying the matrix product $\prod\limits_{k=1}^{s}(W(\tau_k) \otimes I_n)$, whose behavior has been characterized in~\cite{doi:10.1137/060657005}. In particular, the convergence is closely related to the joint strong rootedness of the graph sequence $\mathcal{G}_d(\tau_k)$, as formalized below.

\begin{proposition}
    Let $S(t)$ be a row-stochastic matrix such that the graph induced by each $S(t)$ possesses a strongly rooted structure. Then, there exists a non-negative vector $\mathbf{c}$, with $\mathbf{1}^T \mathbf{c} = 1$, determined by the matrix sequence $\cdots S(t) \cdots S(2)S(1)$, such that:
$\lim\limits_{t\to \infty} (S(t)\otimes I) \cdots (S(1) \otimes I) = \mathbf{1}\mathbf{c}^T \otimes I$,
and the convergence is exponential.
\end{proposition}

This proposition is a direct consequence of the results in~\cite{doi:10.1137/060657005}, and it guarantees the exponential convergence of the system in \eqref{eq:普通共识系统}. It further indicates that the final consensus state is influenced by the initial projection error $\mathbf{e}_2(\tau_0)$, and the evolution of each agent’s state lies within $\text{Row}_\perp(A)$. If the initial state $x_i(t_{i,0})$ is selected such that $e_i(t_{i,0}) = 0$, the effect of \eqref{eq:普通共识系统} can be neglected. In this case, convergence of the projected consensus system \eqref{eq:投影共识系统} ensures that the overall system converges to the solution $x^*$. A natural and simple choice is to set $x_i(t_{i,0}) \in \text{Row}(A)$, for instance, $x_i(t_{i,0}) = 0$.

\subsubsection{Convergence of the Algorithm}

Thus far, we have established two major propositions regarding system convergence. To rigorously derive the convergence of the proposed algorithm in a probabilistic framework, a natural and widely accepted strategy is to construct specific convergence events that fulfill the conditions of these propositions, and then invoke probabilistic arguments to show that such events occur infinitely often with probability one. Toward this end, we introduce a foundational result from probability theory that underpins the subsequent analysis:

\begin{lemma}[Borel–Cantelli]
(i) If a sequence of events $\{A_n\}$ satisfies $\sum\limits_{n=1}^\infty P(A_n) < \infty$, then almost surely only finitely many of them occur, that is:
\[P\left( \limsup_{n \to \infty} A_n \right) = 0.\]

(ii) If the events $\{A_n\}$ are mutually independent and $\sum\limits_{n=1}^\infty P(A_n) = \infty$, then almost surely infinitely many of them occur, that is:
\[P\left( \limsup_{n \to \infty} A_n \right) = 1.\]
\end{lemma}

In the context of this analysis, we apply only part (ii) of the lemma. That is, if the event sequence $\{A_n\}$ is independent and the probability sum diverges, then with probability one, the events occur infinitely often.

\begin{proof}[The Proof of Theorem 2]
The central idea of the proof is to construct error-reducing events and show that these events occur infinitely often with probability one. Two critical components underlie this approach. First, using the assumption $P\big(c(\infty)\big) = 0$, we show that there exists a finite integer $l$ such that the graph sequence $\mathcal{G}_d(\tau_k), \cdots, \mathcal{G}_d(\tau_1)$ is $l-$strongly connected, thereby allowing the formulation of connectivity events. Second, we identify a constant $\kappa$ and construct events in which the projection matrices satisfy $P_i^A(\tau_k), \cdots, P_i^A(\tau_0) \in \mathcal{P}^i_\kappa$. The intersection of these two types of events forms a composite event, whose infinite occurrence can be guaranteed using \textbf{Lemma 8}.

To begin, we provide an explicit result regarding the convergence of system \eqref{eq:投影共识系统}. Since the event $\mathcal{C}(\infty)$ occurs with zero probability, we have:
\[
\bigcup_{l=1}^{\infty} \mathcal{C}(l) = \left( \bigcap_{l=1}^{\infty} \mathcal{C}(l)^c \right)^c = \mathcal{C}(\infty)^c, \text{ which implies } P\left( \bigcup_{l=1}^{\infty} \mathcal{C}(l) \right) = 1.
\]

Thus, with probability one, there exists a finite integer $l$ such that the event $\mathcal{C}(l)$ occurs. As a result, the $l-$strong connectivity of the graph sequence $\mathcal{G}_d(\tau_k), \cdots, \mathcal{G}_d(\tau_1)$ is almost surely guaranteed, completing the first part of the argument.

Let $\eta = \max\{m_1, m_2, \cdots, m_N\}$, and define $\mathcal{T}_i(\tau_s) = \{ t_{i,k} : \tau_s < t_{i,k} \leq \tau_s + \eta \, \overline{T} \}$. According to \eqref{eq:communication-assumption}, $\mathcal{T}_i(\tau_s)$ is a finite set, and agent $i$ will select each block of $A_i$ exactly $\eta$ times over this interval. Define $\mathcal{T}(\tau_k) = \bigcup\limits_{i=1}^N \mathcal{T}_i(\tau_k)$; for any $\tau_k$, we have $|\mathcal{T}(\tau_k)| < \infty$, implying the existence of a positive integer $\kappa$ such that $|\mathcal{T}(\tau_k)| \leq \kappa$. Let $\mathcal{P}_i(\tau_s)$ denote the sequence of random orthogonal projection matrices chosen by agent $i$ over $\mathcal{T}(\tau_s)$, and define the event $\mathcal{J}(\tau_s) = \{\forall i, \mathcal{P}_i(\tau_s) \in \mathcal{P}_\kappa \}$. Since block selections are made independently across agents and time, we have:
\[\mathbb{P}\left( \mathcal{J}(\tau_s) \right) = \prod\limits_{i=1}^N \frac{\eta!}{m_i!(\eta - m_i)!}.\]
This expression provides a lower bound on the probability of $\mathcal{J}(\tau_s)$, denoted as $p$. This concludes the second step.

Let $\varrho$ be the least common multiple of $\kappa$ and $(d+1)l$, $\omega = N\varrho$, and $\vartheta = (N(N-1) + 2)\varrho$. According to \textbf{Proposition 2}, the event
\[\mathcal{H}(n) = \left\{ \|\mathbf{e}_1(\tau_{n\vartheta})\|_2^2 \leq \lambda \|\mathbf{e}_1(\tau_{(n-1)\vartheta+1})\|_2^2 \right\}\]
occurs whenever
\[\mathcal{B}(n) = \left\{ \mathcal{J}(\tau_{(n-1)\vartheta}), \mathcal{J}(\tau_{(n-1)\vartheta+\omega}), \cdots, \mathcal{J}(\tau_{(n-1)\vartheta+(N-1)\omega}) \right\}\]
holds. Since the projection selection process is independent of communication, we have:
\[
\begin{aligned}
    \mathbb{P}(\mathcal{H}(n) \mid \mathcal{C}(l)) &= \mathbb{P}(\mathcal{B}(n) \mid \mathcal{C}(l)) =\\
    & \prod_{i=0}^{N-1} \mathbb{P}(\mathcal{J}(\tau_{(n-1)\vartheta+i\omega})) \geq p^N.
\end{aligned}
\]
Therefore, $\mathbb{P}(\mathcal{H}(n) \mid \mathcal{C}(l)) \geq p^N > 0$, and the sum $\sum\limits_{n=0}^\infty \mathbb{P}(\mathcal{H}(n) \mid \mathcal{C}(l)) = \infty$. By \textbf{Lemma 8}, it follows that:
\[\mathbb{P}\left( \limsup_{n \to \infty} \mathcal{H}(n) \mid \mathcal{C}(l) \right) = 1.\]
This implies that conditioned on $\mathcal{C}(l)$, the event $\mathcal{H}(n)$ occurs infinitely often with probability one. Let $\mathcal{A} = \left\{ \lim\limits_{k \to \infty} \|\mathbf{e}_1(\tau_k)\|_2^2 = 0 \right\}$. Then:
\[\mathbb{P}(\mathcal{A} \mid \mathcal{C}(l)) = 1 \quad \Rightarrow \quad \mathbb{P}(\mathcal{A}^c \cap \mathcal{C}(l)) = 0.\]
Thus,
\[\mathbb{P}(\mathcal{A}) \geq \mathbb{P}(\mathcal{A} \cap \bigcup\limits_{l=1}^\infty \mathcal{C}(l)) = \mathbb{P}\left( \bigcup\limits_{l=1}^\infty \mathcal{C}(l) \right) = 1,\]
which implies $\mathbb{P}(\mathcal{A}) = 1$. Hence, system \eqref{eq:投影共识系统} converges almost surely.

To further establish the convergence of system \eqref{eq:普通共识系统}, define
\[
\sigma = [(d+1)l](N-1) + dl,\quad S_n = W(\tau_{n\sigma}) \cdots W(\tau_{(n-1)\sigma+1}).
\]
From \textbf{Lemma 6}, the graph induced by \( S_n \) possesses a strongly rooted structure. Consequently, there exists a vector \( \mathbf{c} \in \mathbb{R}^n \) such that
\[
\lim_{n \to \infty} (S_n \otimes I) \cdots (S_1 \otimes I) = \mathbf{1} \mathbf{c}^T \otimes I.
\]
Letting \( n \) be the integer quotient of \( k \) divided by \( \sigma \), we obtain
\[
\mathbf{e}_2(\tau_{k+1}) = \left[ (W(\tau_k) \cdots W(\tau_{n\sigma+1}) S_n \cdots S_1) \otimes I_n \right] \mathbf{e}_2(\tau_0).
\]
As \( k \to \infty \), it follows that \( n \to \infty \), and since
\[
W(\tau_k) \cdots W(\tau_{n\sigma+1}) \mathbf{1} \mathbf{c}^T = \mathbf{1} \mathbf{c}^T,
\]
we conclude that there exists \( x_0 \in \text{Row}_\perp(A) \) such that
\[
\lim_{k \to \infty} \mathbf{e}_2(\tau_k) = x_0 \otimes \mathbf{1},
\]
with exponential convergence rate.

Next, we derive the exponential convergence rate of the full algorithm. Note that
\[
\|\mathbf{P}(\tau_s)(W(\tau_s) \otimes I_n)\mathbf{P}(\tau_{s-1})\|_A \leq 1,
\]
implying
\[
\|\mathbf{e}_1(\tau_k)\|_2^2 \leq \|\mathbf{e}_1(\tau_s)\|_2^2, \quad \forall k \geq s.
\]
Since the event \( \bigcup_{l=1}^\infty \mathcal{C}(l) \) occurs almost surely, we can fix \( \vartheta \) such that
\[
\begin{aligned}
& \mathbb{E}\left[ \|\mathbf{e}_1(\tau_{n\vartheta})\|_2^2 \right] 
= \mathbb{E}\left[ \|\mathbf{e}_1(\tau_{n\vartheta})\|_2^2 \mid \mathcal{B}(n) \right] \mathbb{P}(\mathcal{B}(n)) + \mathbb{E}\left[ \|\mathbf{e}_1(\tau_{n\vartheta})\|_2^2 \mid \mathcal{B}(n)^c \right] (1 - \mathbb{P}(\mathcal{B}(n))) \\
&\leq \left(1 - (1 - \gamma^2) \mathbb{P}(\mathcal{B}(n))\right) \mathbb{E}\left[ \|\mathbf{e}_1(\tau_{(n-1)\vartheta})\|_2^2 \right]\leq \left(1 - (1 - \gamma^2)p^N \right)^n \mathbb{E}\left[ \|\mathbf{e}_1(\tau_0)\|_2^2 \right].
\end{aligned}
\]
Letting \( \mu = 1 - (1 - \gamma^2)p^N < 1 \), and for any \( k \geq \vartheta \), it follows that
\[
\mathbb{E}\left[ \|\mathbf{e}_1(\tau_k)\|_2^2 \right] \leq \mu^{\left(\frac{k}{\vartheta} - 1\right)} \mathbb{E}\left[ \|\mathbf{e}_1(\tau_0)\|_2^2 \right].
\]
Let
\[
\tilde{\mu} = \mu^{\frac{1}{(N(N-1)+1)\varrho}},\quad c = \frac{1}{\mu},
\]
then
\[
\mathbb{E}\left[ \|\mathbf{e}_1(\tau_k)\|_2^2 \right] \leq c \tilde{\mu}^k \mathbb{E}\left[ \|\mathbf{e}_1(\tau_0)\|_2^2 \right].
\]

In conclusion, both the standard consensus system \eqref{eq:普通共识系统} and the projected consensus system \eqref{eq:投影共识系统} converge exponentially. Specifically,
\[
\mathbf{e}_2(\tau_k) \to x_0 \otimes \mathbf{1}, \quad \text{for some } x_0 \in [\text{Row}_\perp(A)]^{dN}.
\]
Let each agent’s state decompose as \( x_i = x_{i1} + x_{i2} \), and the solution \( x^* = x_1^* + x_2^* \), where
\[
x_{i1}, x_1^* \in \text{Row}(A),\quad x_{i2}, x_2^* \in \text{Row}_\perp(A).
\]
Since \( x^* \) is the minimum-norm solution, we must have \( x_1^* = x^*, x_2^* = 0 \). Thus,
\[
\lim_{k \to \infty} x_{i2}(\tau_k) = x_0,\quad \lim_{k \to \infty} x_{i1}(\tau_k) = x^*,
\]
and
\[
\lim_{k \to \infty} x_i(\tau_k) = x^* + x_0,
\]
with exponential rate. If the initial states satisfy \( x_i(t_{i,0}) \in \text{Row}(A) \), then \( x_0 = 0 \), and the convergence is to \( x^* \).
\end{proof}

\subsection{Proof of Theorem 3}

Since the augmented system \eqref{eq:augmented-system} is always consistent, the convergence results can be directly established by invoking the proof of \textbf{Theorem 2}. Accordingly, we can redefine the error vectors and block projection matrices, and express the system in a form analogous to \eqref{eq:未分解的系统}.

Let \( z^* = (A')^\dagger b \). For any \( \tau_s \in \mathcal{T} \), define the local error \( \tilde{e}_i(\tau_s) = z_i(\tau_s) - z^* \), and stack them into a global error vector:
\[
\tilde{e}(\tau_s) = \left[ \tilde{e}_1(\tau_s)^T, \cdots, \tilde{e}_N(\tau_s)^T \right]^T \in \mathbb{R}^{N(n+m)},
\]
and the delayed error sequence:
\[
\tilde{\mathbf{e}}(\tau_s) = \left[ \tilde{e}^T(\tau_s), \cdots, \tilde{e}^T(\tau_{s-d}) \right]^T \in \mathbb{R}^{dN(n+m)}.
\]
Let \( \widetilde{P}_i(t_{i,k}) = I_{n+m} \) be the identity projection for the augmented variable. Then the update rule \eqref{eq:inconsistent-update-explicit} can be reformulated as:
\begin{equation}
    \tilde{\mathbf{e}}(\tau_{k+1}) = \widetilde{\mathbf{P}}(\tau_k) \left(W(\tau_k) \otimes I_n\right) \widetilde{\mathbf{P}}(\tau_{k-1}) \tilde{\mathbf{e}}(\tau_k).
    \label{eq:系统-不相容情形}
\end{equation}

By directly applying \textbf{Theorem 2} to \eqref{eq:系统-不相容情形}, the convergence of \eqref{eq:inconsistent-update-explicit} is immediately established. If the initial values are set as \( z_i(t_{i,0}) = 0 \), then each agent’s state converges to the minimum-norm solution \( \tilde{x}^* \).

Let the singular value decomposition of matrix \( A \) be given by
\[
A = U \Sigma V^T,
\]
where $U\in R^{m\times m}, V\in R^{n\times n}$ are orthogonal matrices, 
$\Sigma= \begin{pmatrix} 
\Sigma_r & \ 0 \\ \\ 
         0 & \ 0
 \end{pmatrix}$ with $\Sigma_r=\mathrm{diag}(\sigma_1, \cdots, \sigma_r)$ and \( r = \mathrm{rank}(A) \), the Moore–Penrose pseudoinverse of the augmented system matrix \( (A')^\dagger \) can be expressed as:
\[
(A')^\dagger =
\begin{pmatrix}
    V &  \\
     & U
\end{pmatrix}
\begin{pmatrix}
    \lambda^{-2}(I_n + \lambda^{-2} \Sigma^T \Sigma)^{-1} \Sigma^T  \\
    I_m - \lambda^{-1} \Sigma (I_n + \lambda^{-2} \Sigma^T \Sigma)^{-1} \Sigma^T 
\end{pmatrix}
U^T.
\]
Therefore, the regularized solution \( \tilde{x}^* \) is given by:
\[
\tilde{x}^* = V (I_n + \Sigma^T \Lambda^{-2} \Sigma)^{-1} \Sigma^T \Lambda^{-2} U^T b.
\]

Let \( \widetilde{\Sigma}^\dagger = \lambda^{-2} (I_n + \lambda^{-2}  \Sigma^T \Sigma)^{-1} \Sigma^T \), then
$ \widetilde{\Sigma}^\dagger =
\begin{pmatrix} 
\widetilde{\Sigma}_r & \ 0 \\ \\ 0 & \ 0 
\end{pmatrix} 
$
with $\widetilde{\Sigma}_r=\mathrm{diag}\left(\frac{\sigma_1}{\sigma_1^2 + \lambda^2}, \cdots, \frac{\sigma_r}{\sigma_r^2 + \lambda^2}\right).$

Now, define  matrix
$
\Sigma_0 = \begin{pmatrix} 
\widetilde{\Sigma}_0  & \ 0 \\ \\ 0 & \ 0 
\end{pmatrix} 
$  with $\widetilde{\Sigma}_0=\mathrm{diag}\left(\frac{\lambda^2}{\sigma_1^2 + \lambda^2}, \cdots, \frac{\lambda^2}{\sigma_r^2 + \lambda^2}\right),
$  then the absolute error can be estimated as:
\[
\begin{aligned}
\|x^* - \tilde{x}^*\| &= \left\| V \left( \Sigma^\dagger - \widetilde{\Sigma}^\dagger \right) U^T b \right\| 
= \left\| \Sigma_0 A^\dagger b \right\| \\
&\leq \frac{1}{\left( \frac{\sigma_{\min}}{\lambda} \right)^2 + 1} \|x^*\|_2,
\end{aligned}
\]
where \( \sigma_{\min} \) is the smallest non-zero singular value of \( A \). This yields the relative error bound:
\[
\frac{\|x^* - \tilde{x}^*\|}{\|x^*\|} \leq \frac{1}{\left( \frac{\sigma_{\min}}{\lambda} \right)^2 + 1}.
\]

\bibliographystyle{elsarticle-num} 
\bibliography{ref} 

\end{document}